\documentclass[11pt, reqno]{amsart}
\usepackage{amsmath, amsthm, amscd, amsfonts, amssymb, latexsym, graphicx, color, stmaryrd,enumerate,mathtools}
\usepackage{mathrsfs}
\usepackage[T1]{fontenc}
\usepackage{float} 
\usepackage[all,cmtip]{xy}	
\usepackage[bookmarksnumbered, colorlinks, plainpages]{hyperref}
\usepackage[babel=true]{csquotes}
\usepackage{slashed}
\usepackage{tikz-cd}

\def\presuper#1#2%
  {\mathop{}%
   \mathopen{\vphantom{#2}}^{#1}%
   \kern-\scriptspace%
   #2}

\tikzset{%
    symbol/.style={%
        ,draw=none
        ,every to/.append style={%
            edge node={node [sloped, allow upside down, auto=false]{$#1$}}}
    }
}

\newcommand{\Rr}{\mathbb{R}}
\newcommand{\Cc}{\mathbb{C}}
\newcommand{\Zz}{\mathbb{Z}}

\newcommand{\B}{\mathcal{B}}
	
\newcommand{\Ell}{{\mathrm{\mathcal{F}Ell}}}	
\newcommand{\EllV}{\presuper{\V}{\mathrm{\mathcal{F}Ell}}}	
\newcommand{\FEllV}{\EllV}
\newcommand{\F}{\mathcal{F}}
\newcommand{\Q}{\mathcal{Q}}
\newcommand{\cS}{\mathcal{S}}

\newcommand{\K}{\mathcal{K}}	
\renewcommand{\H}{\mathcal{H}}	
\renewcommand{\L}{\mathcal{L}}	
\newcommand{\C}{\mathcal{C}}	
\newcommand{\Cl}{\mathrm{Cl}}	
\newcommand{\E}{\mathcal{E}}	
\newcommand{\A}{\mathcal{A}} 	
\newcommand{\U}{\mathcal{U}}  	
	
\newcommand{\Tau}{\mathcal{T}}

\newcommand{\Spinc}{\mathrm{Spin^c}}

\newcommand{\Vphi}{\mathrm{\presuper{\varphi}{\V}}}
\newcommand{\Aphi}{\mathrm{\presuper{\varphi}{\A}}} 

\newcommand{\SigmaA}{\mathrm{\Sigma_{\A}}}
\newcommand{\rhophi}{\mathrm{\presuper{\varphi}{\varrho}}}

\newcommand{\KV}{\mathrm{\presuper{\V}{\mathrm{K}}}}
\newcommand{\EV}{\mathrm{\presuper{\V}{\mathrm{E}}}}

\newcommand{\EVgeo}{\EV^{\mathrm{geo}}}	
\newcommand{\KVgeo}{\KV^{\mathrm{geo}}}	

\newcommand{\rangephi}{\mathrm{\presuper{\varphi}{r}}}
\newcommand{\sourcephi}{\mathrm{\presuper{\varphi}{s}}}

\newcommand{\Sblup}{\mathrm{SBlup}}
\newcommand{\SBlup}{\Sblup}

\renewcommand{\H}{\mathcal{H}} 
\newcommand{\G}{\mathcal{G}}	
\newcommand{\Gad}{\G^{\mathrm{ad}}}	

\newcommand{\M}{\mathcal{M}}		

\renewcommand{\L}{\mathcal{L}} 

	% ...

\newcommand{\V}{\mathcal{V}}		
		% ...
\newcommand{\End}{\mathrm{End}}		
\newcommand{\Hom}{\mathrm{Hom}}		
\renewcommand{\P}{\mathcal{P}}		
\newcommand{\Diff}{\mathrm{Diff}}	

\newcommand{\GF}{\mathrm{\G_{\F}}}
\newcommand{\MF}[1][]{\mathrm{M}^{\F}_{#1}}

\newcommand{\pd}{\mathrm{pd}}

\newcommand{\im}{\operatorname{im}} 
	
\newcommand{\ind}{\operatorname{ind}}

\newcommand{\Id}{\operatorname{Id}}	

\newcommand{\pr}{\mathrm{pr}}
\newcommand{\odd}{\mathrm{odd}}
\newcommand{\nc}{\mathrm{nc}}
\newcommand{\ad}{\mathrm{ad}}
\newcommand{\ev}{\mathrm{ev}}

\newcommand{\iso}{\xrightarrow{\sim}}	

\newcommand{\Gphi}{\mathrm{\presuper{\varphi}{\G}}} 

\newtheorem{Thm}{Theorem}[section]
\newtheorem{Lem}[Thm]{Lemma}
\newtheorem{Prop}[Thm]{Proposition}

\theoremstyle{definition}
\newtheorem{Def}[Thm]{Definition}
\newtheorem{Ex}[Thm]{Example}

\newtheorem{Rem}[Thm]{Remark}

\begin{document}
\setcounter{page}{1}


\title[$K$-homology for Lie manifolds]{A geometric approach to  $K$-homology for Lie manifolds}

\author[Karsten Bohlen, Jean-Marie Lescure]{Karsten Bohlen, Jean-Marie Lescure}





\maketitle

\begin{abstract}
We show that the computation of the Fredholm index of a fully elliptic pseudodifferential operator on an integrated Lie manifold can be reduced to the computation of the index of a Dirac operator, perturbed by a smoothing operator, canonically associated, via the so-called clutching map. To this end we adapt to our framework ideas coming from Baum-Douglas geometric $K$-homology and in particular we introduce a notion of geometric cycles, that can be categorized as a variant of the famous geometric $K$-homology groups, for the specific situation here. We also define a comparison map between this geometric $K$-homology theory and a relative $K$-theory group, directly associated to a fully elliptic pseudodifferential operator.
 \end{abstract} 

\begin{quote}
\footnotesize{{\sc R\'esum\'e.} Nous d\'emontrons que le calcul de l'indice de Fredholm d'un op\'erateur pseudodiff\'erentiel pleinement elliptique sur une vari\'et\'e de Lie int\'egr\'ee peut \^etre ramen\'e \`a celui de l'indice d'un op\'erateur de Dirac perturb\'e par un op\'erateur r\'egularisant et canoniquement associ\'e via l'application de serrage (clutching). 
Pour cela nous adaptons \`a notre situation des id\'ees venant de la $K$-homologie g\'eom\'etrique de Baum-Douglas, en particulier nous introduisons une notion de cycles g\'eom\'etriques qui engendrent, dans le contexte sp\'ecifique de l'article, un analogue de la $K$-homologie g\'eom\'etrique bien connue. Nous d\'efinissons \'egalement une application de comparaison entre notre  $K$-homologie g\'eom\'etrique et un groupe de $K$-th\'eorie relative directement associ\'e aux op\'erateurs pseudodiff\'erentiels pleinement elliptiques.}
\end{quote}

\section{Introduction}

The Atiyah-Singer index theorem is a celebrated and fundamental result with numerous applications in topology, geometry and analysis. Atiyah and Singer proved the index theorem for elliptic pseudodifferential operators on compact manifolds using $K$-theory and pseudodifferential operator theory. Later on Atiyah, Bott and Patodi \cite{ABP} proved the index formula for Dirac type operators using heat kernel methods, and this approach revealed itself very useful for the case of closed smooth manifolds \cite{APS1, APS2, APS3, BGV, BS, MelroseBook}.  To recover the full Atiyah-Singer theorem from the special cases covered by heat kernel methods, one needs to reduce the index problem for arbitrary elliptic operators to the one for Dirac type operators. Such a reduction is a byproduct of the approach to index theory due to P. Baum and R. Douglas, cf. \cite{BD} and \cite{BD2}. They constructed a geometric $K$-homology and a suitable comparison homomorphism between the geometric and analytic $K$-homology groups, and the complete proof that this is in fact an isomorphism was published recently \cite{BHS}. 
Thanks to the Baum-Douglas approach to the index theorem, the computation of the Fredholm index of an elliptic pseudodifferential operator on a compact closed manifold can be reduced to the computation of the index of a suitable geometric Dirac operator, naturally associated to a geometric cycle.
Motivated by the extension of heat kernel methods to Lie manifolds \cite{BS}, the purpose of the present work is to address the question of reducing the index problem to the one of Dirac operators for the singular manifolds for which a suitable Lie groupoid allows to well pose the index problem. More precisely, we consider (integrated) Lie manifolds $(M,\G)$, that is  amenable Lie groupoids $\G$ over  compact manifolds with corners such that $M_{0}=M\setminus \partial M$ is saturated and $\G_{M_{0}}=M_{0}\times M_{0}$. For instance, that covers  the following situations: 
\begin{itemize}
	\item Manifolds with corners \cite{M}. Here $\G = \G_b$ is obtained after blowing-up successively the submanifolds $H_i\times H_i$ into $M\times M$, where $H_i $ run through the connected boundary hypersurfaces of $M$, and then removing the so-called lateral faces in $b$-geometry terminology, which equivalently amounts to consider the subspace $\mathrm{\Sblup}_{r,s}(M^2,(H_i^2)_{{}_{i}})$ of the blow-up according to the terminology of \cite{DS}.
	\item  Manifolds with fibered corners, and thus equivalently stratified pseudomanifolds \cite{DLR}. Here $\G=\G_\pi$ is obtained as before (blowing-up and removing the lateral faces), but now starting with $\G_b$ in which the fibred diagonals $H_i\times_{\pi} H_i$ are blown-up in the order prescribed by the order relation between boundary hypersurfaces. 
	\item Manifolds with amenable foliated boundary. The pseudodifferential operators are studied in \cite{R} and the corresponding groupoid $\G_\F$, although not directly used, is constructed. Actually, it is stated in \cite{DS} that $\G_\F$ is obtained by blowing-up in $M\times M$ the holonomy groupoid of the foliation on the boundary, that is $\G_\F = \mathrm{\Sblup}_{r,s}(M\times M, \mathrm{Hol}(\F))$. 
	\item There are many other examples related to singular spaces, see for instance \cite{N,N2015,CNQ2018}.
\end{itemize}
In such a case, there is a well defined notion of full ellipticity for operators in the corresponding calculus, that ensures the Fredholmness of the associated operators on $M$. The question can now be made more precise. Given a fully elliptic operator $P$ on $(M,\G)$, can we construct a Dirac operator $D$ in the same calculus, which is Fredholm and with the same index as $P$? Contrary to the case of $C^{\infty}$ compact manifolds without boundary, we are not able to give an affirmative answer to this question. Nevertheless, we are able to solve positively the question by allowing tamed Dirac operators, that is, Dirac operators perturbed by smoothing elements in the calculus. Along the way, we prove that if there is no obstruction at the level of $K$-theory for the full ellipticity (and thus Fredholmness) of Dirac operators, then a perturbation into a Fredholm operator using smoothing operators and elementary Dirac operators always exists. This echoes previous works by Bunke \cite{B2009} and Carrillo Rouse-Lescure \cite{CRL}. Explicitly, our main result is the following, cf. Theorem \ref{Thm:redDirac}: Given a pseudodifferential operator $P$ in the pseudodifferential calculus of the integrated Lie manifold $(M, \G)$, there is a Callias type operator $\C$ on the clutching Lie manifold $(\Sigma_{\A}, \Gphi)$, contained in the relevant calculus, such that the appropriate stable homotopy classes, and thereby the Fredholm indices, of $\C$ and $P$ are equal. The result facilitates a reduction of the index problem for the pseudodifferential operator $P$ to the index computation of the simpler geometric operator $\C$. We note that the literature consists of various index formulas in specialized cases of specific Lie manifolds and for geometric operators of this simpler type. More generally, in conjunction with the recent work of Bohlen-Schrohe \cite{BS}, the reduction of the index problem to first order geometric operators furnishes
a corresponding index formula for fully elliptic pseudodifferential operators on Lie manifolds. Along the way to our main result we prove numerous auxiliary results that are of independent interest.  Also, these considerations bring a notion of geometric cycles that mimics the original one of Baum-Douglas with the following main variations into the choice of ingredients:
\begin{enumerate}
\item We only accept submersions of manifolds with corners $\varphi :\Sigma \to M$ instead of general continuous maps from  $\Spinc$ manifolds to $M$;
\item We replace Dirac operators with tame Dirac operators.
\end{enumerate}
The point (1) is apparently rather restrictive but actually sufficient for our purpose, since the geometric cycles constructed by the clutching process are of this kind. 
Furthermore, the classical operations on geometric cycles: isomorphisms, direct sums, cobordisms and vector bundle modifications have a natural analog here. The resulting abelian group, that we call geometric $K$-homology of $(M,\G)$ can then be compared with the suitable relative $K$-theory group (also known as stable homotopy group of fully elliptic operators). Our approach is definitely tied to the clutching construction which, in a sense, dictated to us the most convenient notion of geometric cycles for our purpose, the latter being, as explained above, the possibility of representing the index of abstract pseudodifferential operators by the one of Dirac operators, perturbed by smoothing operators. We mention that other and more general approaches of geometric $K$-homology for groupoids exist, in particular in \cite{CRW,Connesbigbook,EM}. The observation that one needs to add a regularizing operator to the Dirac operator in order to be able to represent all $K$-homology classes has appeared before, e.g. in \cite{B2009, B1995, MelrosePiazza, Nistor}, in work on cylindrical ends. In that case, one can absorb the regularizing perturbation into the operator with a suitable change of weight. This approach then leads to obstructions to the existence of analogs of the Atiyah-Patodi-Singer boundary conditions on manifolds with higher dimensional corners, as in \cite{Nistor}. 
To finish this introduction, let us observe that most bivariant $K$-theory groups used in this article are of the form $KK_*(\Cc,A)$. Therefore, the reader unfamiliar with Kasparov bivariant $K$-theory should keep in mind that only usual $K$-theory is really needed here. However, it is almost always more convenient for us to have $K$-theoretic elements represented by Kasparov bimodules rather than by idempotents or unitaries, and hence, unless otherwise stated, $K_*(A)$ will be identified with $KK_* (\Cc,A)$.

\subsection*{Overview}
The article is organized as follows. 

\begin{itemize}
\item In Section \ref{Section:2}, we study fully elliptic operators contained in the pseudodifferential calculus on the tuple $(M, \G)$, where $\G$ is a Lie groupoid over $M$. We introduce the group of stable homotopy classes of fully elliptic operators $\FEllV(M)$ which is defined to equal the relative $K$-theory group $K(\mu)$, where $\mu$ is the homomorphism of the continuous functions $M$ into the full symbol algebra, given by the action as multiplication operators. We prove a Poincar\'e duality type result which states that the Fredholm index can be expressed in a precise way in terms of an index map defined in terms of purely geometric data, given by suitable deformation groupoids. More precisely, we recall the geometric model for the full symbol space of a pseudodifferential operator on a Lie groupoid in terms of the deformation groupoid $\Tau$, referred to as the \emph{noncommutative tangent bundle}. 
The $K$-theory group $K_0(C^{\ast}(\Tau))$ turns out to be the natural receptable for the stable homotopy class of the full symbol of a fully elliptic pseudodifferential operator. Indeed, we show that there is a very general weak form of Poincar\'e duality, taking the form of an isomorphism of groups $\pd : \FEllV(M) \overset{\cong}{\longrightarrow} K_0(C^{\ast}(\Tau))$ induced by a non-commutative principal symbol homomorphism $\sigma_{\mathrm{nc}}$.

\item In Section \ref{tameDirac}, we investigate so-called \emph{tamings} of geometric Dirac operators on Lie groupoids. The main result is a \emph{Diracification theorem} which states that any fully elliptic pseudodifferential operator $P$ on a Lie groupoid, whose principal symbol has the same class in $K$-theory as a given Dirac operator $D$, has a so-called taming $B = (D \oplus D') + R$, consisting of a Dirac operator $D'$, the principal symbol of which being $0$ in $K$-theory, and a smoothing operator $R$, such that the classes of the operators $B$ and $P$ agree in the relative $K$-theory group $K(\mu)$. 
We introduce a pushforward operation on the level of deformation groupoids and show that this operation commutes with the Fredholm index. 

\item In Section \ref{Kgeo}, we introduce geometric $K$-homology groups $\KVgeo(M)$ (of even and odd degree) on a given integrated Lie manifold $(M, \G)$. Using the pushforward operation, introduced in the previous section, we define the comparison map $\lambda \colon \KVgeo(M) \to K_0(C^{\ast}(\Tau))$. The main result of this section is a proof that the comparison map $\lambda$ is a well-defined group homomorphism. This result is correctly viewed as a generalization of the cobordism invariance of the analytic index in the standard setting.

\item In Section \ref{RedDirac}, we define the so-called \emph{clutching construction} in the setting of Lie manifolds. The resulting quotient map $c \colon K_0(\mu) \to \KVgeo(M)$ is 
induced by the map associating to a full symbol of a (fully elliptic) pseudodifferential operator a geometric cycle.
We show that $c$ is a well-defined homomorphism of groups. 
Then we finish the proof of the main result of the paper, which states that the computation of the Fredholm index of any fully elliptic pseudodifferential operator on a Lie manifold can be reduced to the computation of the index of a geometric Callias type operator, associated to the geometric data, generating the geometric $K$-homology group. 
\end{itemize}

The view of index theory on Lie manifolds considered in this work can be summarized in terms of the following commutative diagram:
\begin{center}
\begin{equation} \label{Main}
\begin{tikzcd}[row sep=huge, column sep=huge, text height=1.5ex, text depth=0.25ex]
& K_0(C^{\ast}(\Tau)) \arrow{rr}{\ind_{\F}} & &  \Zz  \\
\KVgeo(M) \arrow{ur}{\lambda} & & \arrow{ll}{c} K_0(\mu)=\FEllV(M)  \arrow[ul, "\pd", "\simeq"']\arrow[ur,  "\ind"']
\end{tikzcd}
\end{equation}
\end{center}

The clutching homomorphism $c$ is defined in Section \ref{RedDirac}, the comparison homomorphism is defined in Section \ref{Kgeo}, and the Poincar\'e duality isomorphism $\pd = \pd_{\partial M}$ in Section \ref{Section:2}. The geometrically defined index map $\ind_{\F}$ is constructed in Section \ref{Section:2}, as well as the index $\ind = \ind_{\partial M}$, associated to the relative $K$-theory, cf. Remark \ref{Rem:Fh-F}.

\section{Poincar\'e duality}

\subsection*{Full ellipticity, groupoids and $K$-theory}

\label{Section:2}
Let  $\G\rightrightarrows M$ be a Lie groupoid. 
For simplicity we always assume that $\G$ is amenable, so that the maximal and reduced $C^{*}$-algebraic completion  of $C_{c}(\G)$ is the same and is nuclear   \cite{Re1980,AR}. The Lie algebroid of $\G$ is denoted by $\A$. 

The manifold $M$ will always be compact and may have corners. In that case, we assume that the boundary hypersurfaces of $M$ and $\G$ are embedded  \cite{MelroseMWC}, and the source and range maps are \emph{submersions between manifolds with corners} \cite{LN2001}, also called tame submersions in \cite{N2015}. We recall that it means, for a  $C^{\infty}$ map $f : M \to N$  between manifolds with corners, that at any point $x\in M$ we have:
\begin{equation}\label{eq:tame-submersion}
 df_{x}(T_{x}M) = T_{f(x)}N \quad \text{and}\quad (df_{x})^{-1}(T_{f(x)}^{+}N) = T_{x}^{+}M. 
\end{equation}
Here $TM$ denotes the ordinary tangent vector bundle and $ T^{+}M$ its subset of inward pointing vectors. Under such assumptions, $f$  preserves the  codimension of points and its fibers have no boundary.

 If $E_j\to M$ are vector bundles, we denote by $\Psi^{*}_{\G}(E_{0},E_{1})$  the space of compactly supported $\G$-pseudodifferential operators \cite{C1982,MP,M,NWX}.  The principal symbol map is denoted by:
\[
 \sigma_{\pr} \colon \Psi^{m}_{\G}(E_{0},E_{1}) \to C^{\infty}(S(\A^{\ast}),\overline{\pi}^{\ast}(\Hom(E_0, E_1))).
\]
Here $S(\A^{\ast})$ is the sphere bundle of the dual Lie algebroid $\A^{\ast}$ of $\G$ and $\overline{\pi}$ denotes any of the projection maps of  $\A^{\ast}$ and $S(\A^{\ast})$ onto $M$. 

Let $\overline{\Psi^{m}_{\G}}(E_{0},E_{1})$ denote the closure of $\Psi^{m}_{\G}(E_{0},E_{1}))$ into $\mathrm{Mor}(\H^s(\G,E_0),\H^{s-m}(\G,E_1))$ \cite{Vassout}, where we fix once and for all $s = m$ in the sequel. For  simplicity we denote $\Psi_{\G}(E)=\overline{\Psi^{0}_{\G}}(E)$. We get a short exact sequence of $C^*$-algebras:
\begin{equation}\label{eq:principal-symbol-seq}
\xymatrix{
C^*(\G,\End(E))  \ar@{>->}[r] & \Psi_{\G}(E) \ar@{->>}[r]^-{\sigma_{\mathrm{pr}}} & C(S(\A^*),\overline{\pi}^{\ast}\End(E)).
}
\end{equation}
Now let $F\subset M$ be a closed subspace which is saturated, by which we mean that $s^{-1}(F)=r^{-1}(F)$, and set $O := M \setminus F$. Then $\G_F$ is a continuous family groupoid \cite{Paterson,LMN}. By restricting over $F$, we get the  $F$-indicial symbol map:
 \[
 I_{F}  \colon \Psi^{\ast}_{\G}(E_{0},E_{1})  \to  \Psi^{\ast}_{\G_{F}}(E_{0}|_{F},E_{1}|_{F}).
\]
Gathering both symbol maps, we get the \emph{$F$-joint symbol map}:
\begin{equation}
\sigma_{F,m}=(\sigma_{\pr}, I_{F}): \Psi^{m}_{\G}(E_{0},E_{1}) \longrightarrow  C(S(\A^{\ast}),\overline{\pi}^*\Hom(E_0, E_1)) \times  \overline{\Psi^{m}_{\G_{F}}}(E_{0}|_{F},E_{1}|_{F}).
\end{equation}
The range $\Sigma^{m}_{F}(E_0, E_1)$ of $\sigma_{F,m}$ is called the \emph{$F$-joint symbols space} and its closure is denoted by $\overline{\Sigma^{m}_{F}}(E_0, E_1)$. We will write $\Sigma_F$ for $\overline{\Sigma^{0}_{F}}$. 

 This gives the short exact sequence of $C^*$-algebras \cite{LMN}:
\begin{equation}\label{eq:F-joint-symbol-seq}
\xymatrix{
C^*(\G_O,\End(E))  \ar@{>->}[r] & \Psi_{\G}(E) \ar@{->>}[r]^-{\sigma_{F}} & \Sigma_F(E).
}
\end{equation}
where $C^*(\G_O,\End(E))$ is the closure of $C^\infty_c(G_O,r^*\End(E))$ into $\Psi_{\G}(E)$.
 
We recall that for any elliptic 
$A\in  \Psi^{1}_{\G}(E,E')$, the unbounded operator 
\[
\Lambda_E=(1+A^*A)^{1/2} 
\]
defined by functional calculus belongs to  
 $\overline{\Psi^{1}_{\G}}(E)$, realizes isomorphisms  
 \[
 \H^s(G,E)\iso \H^{s-1}(G,E) 
\] 
 and satisfies $\Lambda^{-1}_E\in C^*(\G,\End(E))$ \cite{LN2001, Vassout}. 
 
\begin{Def}\label{Def:full-ell}
 We  say that $P \in \overline{\Psi^{m}_{\G}}(E)$ is \emph{$F$-fully elliptic} if $\sigma_F(\Lambda^{-m}_EP)\in \Sigma_F(E)^\times$. 
\end{Def}
We now turn our attention to the natural groups constructed out of $F$-full elliptic operators \cite{Karoubi,Sk1983,S2005,AS2011}.  A convenient way to handle them is to use the so-called  $K$-group $K(f)$ of a homomorphism   $f :A\to B$  (\cite[II.2.13]{Karoubi}). We recall the definition: the set of cycles $\Gamma(f)$ is given by triples $(\E_{0},\E_{1},\alpha)$ where the $\E_{i}$ are finitely generated projective $A$-modules and $\alpha : \E_{0}\otimes_{f}B\to  \E_{1}\otimes_{f}B$ is an isomorphism. The notions of isomorphisms ($\simeq$) and direct sums ($\oplus$) in $\Gamma(f)$ are the obvious ones. A cycle $(\E_{0},\E_{1},\alpha)$ is elementary  if $\E_{0}=\E_{1}$ and $\alpha$ is homotopic to the identity within the automorphisms of $ \E_{0}\otimes_{f}B$. Finally, the equivalence relation giving $K(f)=\Gamma(f)/\sim$ is defined by: 
\[ \sigma \sim \sigma' \text{ if there exist elementary } \tau,\tau' \text{ such that } \sigma\oplus\tau \simeq \sigma' \oplus \tau'. \]
Using $K$-theory groups of morphisms is particularly relevant here. Indeed on the one hand and as suggested above, they allow to recover in an obvious way the stable homotopy groups of elliptic operators introduced \cite{S2005}. On the other hand, $K$-theory groups of morphisms are functorial (\cite{AS2011}) and always isomorphic in a natural way to the $K$-theory of a $C^*$-algebra: 
$$K(f)\simeq K_0(C_f)$$
 where $C_f$ is the mapping cone of $f$:
\[
C_{f} := \{(a, F) \in A \oplus C_0([0, 1), B) : F(0) = f(a)\}.
\]
Note that $K(f)\simeq K_0(\ker f)$ when $f$ is onto \cite{AS2011}.
We keep a notation inspired by Savin's one \cite{S2005}.  

\begin{Def}
The group $\Ell_{F}(\G)$ will be defined as the $K$-group $K(\mu_{F})$ of the natural homomorphism  $\mu_{F} : C(M)\longrightarrow \Psi_{\G}/C^*(\G_O)=\Sigma_F$.
\label{Def:Fh}
\end{Def}
We also consider another natural homomorphism $\mu_{\mathrm{pr}} : C(M)\to \Psi_{\G}/C^{*}(\G)\simeq C(S(\A^{*}))$.  If $P\in  \Psi^m_{\G}(E_{0},E_{1})$ is $F$-fully elliptic, $m\ge 0$, then it canonically defines  classes:
   \[  [P]_{F}=\big[\E_{0},\E_{1}, \sigma_{F}(\Lambda_{E_1}^{-m}P)\big] \in K(\mu_{F})=\Ell_{F}(\G), \]
    \[  [P]_{\mathrm{pr}}=\big[\E_{0},\E_{1}, \sigma_{\mathrm{pr}}(\Lambda_{E_1}^{-m}P)\big] \in K(\mu_{\mathrm{pr}})\simeq K^0_c(\A^*), \]
with $\E_j=C(X,E_j)$.  The resulting classes do not depend on the choice of $\Lambda_{E_1}$.  Indeed, consider $\Lambda_A=(1+A^*A)^{1/2}$ and  $\Lambda_B=(1+B^*B)^{1/2}$. Since $C(t)=(1-t)A^*A+tB^*B$ is elliptic self-adjoint and non negative we get a family $\Lambda_t= (1+C(t))^{1/2}\in \Psi_\G^1(E_1)$ of invertible elements connecting $\Lambda_A$ to $\Lambda_B$, and $\sigma_F(\Lambda_t^{-m}P)$ provides the desired homotopy. The same is true for the principal symbol classes. 

Note that there is a natural isomorphism $ C_{\mu_{\mathrm{pr}}}\simeq C_0(\A^*)$ and therefore:
\[ 
   K(\mu_{\mathrm{pr}})\simeq K^0_c(\A^*). 
\]
We will denote by $[P]_{\pr,\ev}$ the image of $[P]_{\pr}$ through this isomorphism.

\begin{Prop}\label{Prop:opposite-in-Krel}
Let  $P_{+}\in\Psi_{\G}(E_{+},E_{-})$ be $F$-fully elliptic. Then  $P_{+}^{*}$ is $F$-fully elliptic and 
\begin{equation}\label{eq:oppposite-in-Krel}
 [P_{+}]_{F} = - [P_{+}^{*}]_{F} \in K_{0}(\mu_{F}).
\end{equation}
\end{Prop}
\begin{proof}
If $P_{+}$ is $F$-fully elliptic, then $P_{-}:=P_{+}^{*}$ is also  $F$-fully elliptic. Set $P=\begin{pmatrix} 0& P_{-} \\ P_{+} & 0 \end{pmatrix}\in \Psi_{\G}(E)$, $E=E_{+}+E_{-}$. We get by polar decomposition of the invertible element $p=\sigma_{F}(P)$ in the $C^{*}$-algebra $\Sigma_{F}(E)$ the existence of $U,T\in \Psi_{\G}(E)$ such that 
 $u=\sigma_{F}(U)$ is  unitary, $\vert p\vert =\sigma_{F}(T)$  and  $ p =u\vert p\vert$. Since $\vert p\vert $ is of degree $0$, since $p$ of degree $1$ with respect to the $\Zz_{2}$-grading and since $\vert p\vert$ is invertible,  the unitary $u$ is necessarily of degree $1$ with respect to the $\Zz_{2}$-grading.
Therefore, we can assume that $U=\begin{pmatrix} 0& U_{-} \\ U_{+} & 0 \end{pmatrix}$ and  $T=\begin{pmatrix} T_{+}& 0 \\ 0 & T_{-} \end{pmatrix}$. Since $p$ is selfadjoint, we  can also assume that the unitary $u$ is self-adjoint, thus   $u_{-}=u_{+}^{*}=u_{+}^{-1}$. Since $p$ is positive, we get a homotopy $p=u\vert p\vert \sim u$, therefore 
\[
[P_{+}]_{F} + [P_{+}^{*}]_{F}=  [\E,\E,p] =  [\E,\E,u] = [\E_{+},\E_{-},u_{+}]+ [\E_{-},\E_{+},u_{+}^{-1}] = 0 \in  K(\mu_{F}). 
\]
\end{proof}

\subsection*{Index and Poincar\'e duality results}

There is an index map coming with $\Ell_F(\G)$. Indeed, the commutative diagram 
 \[
\xymatrix{
 C(M)\ar[d]_{\mu_0} \ar[r]^{\mu_F} & \Sigma_F \ar[d]^{\Id}   \\
 \Psi_\G \ar[r]^{\sigma_F} &  \Sigma_F 
}
\]
gives rise to a homomorphism $K(\mu_F)\to K(\sigma_F)$ and since $\sigma_F$ is onto, we also have a natural isomorphism $K(\sigma_F)\simeq K_0(\ker \sigma_F)=K_0(C^*(\G_O))$ \cite{AS2011}. Their composition is called the ($F$-)index map:  
\begin{equation}\label{eq:Khom-F-ind}
 \ind_F \colon \Ell_F(\G) \to  K(C^*(\G_O)). 
\end{equation}
It is possible to get a slightly more geometrical description of  \eqref{eq:Khom-F-ind} by establishing a kind of Poincar\'e duality with the help of deformation groupoids. Let us introduce the necessary objects. 

The adiabatic groupoid $\Gad \rightrightarrows M_{\mathrm{ad}} := M \times [0,1]$ is the natural Lie groupoid integrating the Lie algebroid
$(\A^{\mathrm{ad}}, \varrho^{\mathrm{ad}})$ given by: $\A^{\mathrm{ad}} = \A \times [0,1]$ and $\varrho^{\mathrm{ad}} \colon \A^{\mathrm{ad}} \to TM \times T[0,1], \ \A^{\mathrm{ad}} \ni (x, v, t) \mapsto (x, tv, t, 0) \in TM \times T[0,1] = T M_{\mathrm{ad}}$. More precisely, 
\[\G^{\mathrm{ad}} = \A  \times \{0\} \cup \G \times (0,1] \text{ and } \A(\G^{\mathrm{ad}}) \cong \A^{\mathrm{ad}},\]
see also \cite{Connesbigbook, HS, NWX}.
Out of the adiabatic groupoid we construct the so-called \emph{$F$-Fredholm groupoid}: 
\begin{equation}\label{eq:F-Fredholm-gpd}
 \G^{\F}_F := \Gad \setminus (\G_{F} \times \{1\}) \rightrightarrows \MF[F] = (M \times [0,1]) \setminus (F\times \{1\})
\end{equation}
This is again a Lie groupoid (as an open subset of $\Gad$). The  \emph{non-commutative tangent bundle}  is defined by:
\begin{equation}\label{eq:F-nc-tangent-space}
\Tau_FM := \G^{\F}_F \setminus (G_O \times (0,1]) \rightrightarrows M_F= \MF[F] \setminus (O \times (0,1]) .
\end{equation}
It is  a $C^{\infty,0}$ groupoid \cite{Paterson}.  The exact sequence: 
\begin{equation}\label{eq:K-equi-GF-Tau}
\xymatrix{
C^{\ast}(\G_O\times (0,1]) \ar@{>->}[r] & C^{\ast}(\GF) \ar@{->>}[r]^{e_0} & C^{\ast}(\Tau_{F}M)
}
\end{equation}
possessing a contractible kernel and a nuclear quotient, we get an isomorphism $e_0: K_0(C^*(\G^\F_F))\to K_0(C^*(\Tau_FM))$, and moreover $e_0\in KK(C^*(\G^\F_F), C^*(\Tau_FM))$ provides a $KK$-equivalence (that is, is invertible).
 Considering the restriction $e_1: C^*(\G^\F_F)\to C^*(\G_O)$, we get another index map:  
\[
\ind^{\F}_F := (e_1)_{\ast} \circ (e_0)_{\ast}^{-1} \colon K_0(C^{\ast}(\Tau_FM)) \to K_0(C^*(G_O)). 
\]
The analogy to the commutative case is that $\Sigma_F$ is the \emph{noncommutative cosphere bundle}, relative to $F$, and $\Tau_FM$ is the \emph{noncommutative tangent bundle}, relative to $F$, associated with $\G\rightrightarrows M$, see also \cite[Section 5]{LMN}.  We end up with a Poincar\'e duality like theorem.

\begin{Thm}
There is a group isomorphism:
\begin{equation}\label{eq:Ell2nct}
\pd_F \colon \Ell_F(\G) \to K_0(C^{\ast}(\Tau_FM))
\end{equation}
 such that $e_0(\pd_F[P]_{F}) = [P]_{\pr,\ev} \in K^0_c(\A^*)$ and
$\ind^{\F}_F(\pd_{F}[P]_{F}) = \ind_F([P]_{F})$. Here $e_0$ is the restriction map $C^{\ast}(\Tau_FM) \to C^{\ast}(\A)$. 
\label{Thm:Fh-F}
\end{Thm}

\begin{proof}
By \cite{Karoubi,Sk1983,S2005} we have a natural isomorphism: 
\begin{equation}\label{eq:Ell-Kcone}
 \Ell_F(\G)=K(\mu_F) \simeq K_0(C_{\mu_F})
\end{equation}
 where $C_{f}$ denotes the mapping cone of $f$. The isomorphism 
$ K_0(C_{\mu_F}) \simeq K_0(C^{\ast}(\Tau_FM)) $ can be proved following verbatim the proof of Theorem 10.6 in \cite[Theorem 10.6]{DLR}. Alternatively, we use the following short argument\footnote{communicated to us by G. Skandalis}. Denote by $\G^{\mathrm{ad}}_0$ the restriction of $\G^{\mathrm{ad}}$ to $M\times[0,1)$.   Since $\Psi_{\G^{\mathrm{ad}}_0}/C^*(\G^{\mathrm{ad}}_0)\simeq C_0(S(\A^*)\times [0,1))$ is contractible, the inclusion  $C^*(\G^{\mathrm{ad}}_0)\to \Psi_{\G^{\mathrm{ad}}_0}$ is an isomorphism in $K$-theory, as well as the inclusion 
\[C^{\ast}(\Tau_FM)= C^*(\G^{\mathrm{ad}}_0)/C^*(\G_O\times (0,1)) \subset \Psi_{\G^{\mathrm{ad}}_0}/C^*(\G_O\times (0,1)).\]
If $\mu_0:C(M)\to \Psi_{\G}$ denotes the natural homomorphism, then we have a homomorphism $\kappa : C_{\mu_0}\to \Psi_{\G^{\mathrm{ad}}_0}$ given by $(f,P)\mapsto P$. Using on the one hand the commutative diagram 
\[
\xymatrix{
 C_0((0,1),\Psi_{\G}) \ar@{=}[d] \ar@{>->}[r] & C_{\mu_0} \ar[d]_{\kappa} \ar@{->>}[r] &  C(M) \ar[d]_{\kappa_0}  \\
 C_0((0,1),\Psi_{\G}) \ar@{>->}[r] &  \Psi_{\G^{\mathrm{ad}}_0} \ar@{->>}[r] & \Psi_{\A}. 
}
\]
and on the other hand that $\kappa_0 : C(M)\longrightarrow \Psi_{\A}$ is a $KK$-equivalence, we get by the five lemma that $\kappa$ is a $KK$-equivalence too. Quotienting both algebras $C_{\mu_0}$ and $\Psi_{\G^{\mathrm{ad}}_0}$ by the ideal $C^*(\G_O\times (0,1))$, we get that the homomorphism 
\[C_{\mu_F}\longrightarrow \Psi_{\G^{\mathrm{ad}}_0}/C^*(\G_O\times (0,1))\]
is a $KK$-equivalence. The remaining assertions in Theorem \ref{Thm:Fh-F} are then easy. 
\end{proof}


\begin{Rem}
Denote by $\partial \colon K_1(\Sigma_F) \to K_0(C^{\ast}(G_{O}))$ the connecting "index" map in the $K$-theory six-term exact sequence associated to the short exact sequence of the full symbol map \eqref{eq:F-joint-symbol-seq}.
By \cite{DL2009} we have that 
\[
\ind_{\F}^F(\pd_F[P]_F) = \partial [\sigma_F(P)]_1
\]

where by $[\sigma_F(P)]_1$ we denote the $K_1$-class of $P \in \Psi_{\G}(E)$ under the full symbol map $\sigma_F$. By Theorem \ref{Thm:Fh-F} we obtain that $\ind_F([P]_F) = \partial[\sigma_F(P)]_1$. In particular for $F = \partial M$ and $\G$ an integrated Lie manifold, the index map $\ind_F$ recovers the Fredholm index. 
\label{Rem:Fh-F}
\end{Rem}

\section{From abstract operators to tame Dirac operators}

\label{tameDirac}

\subsection*{Dirac bundles over groupoids}
Let $\G\rightrightarrows M$ be an amenable Lie groupoid with Lie algebroid $\A$. We equip $\A$ with a euclidean structure and denote by $\Cl(\A)$ the corresponding bundle of Clifford algebras over $M$. We define Dirac bundles as in \cite{ALN}, see also \cite[Definition 3.1]{B2009} and \cite{ALN2}. That is, we replace $TM$  by $\A$ in the definition of $c$, in other words we deal with a $\Cl(\A)$-module complex vector bundle $E$ over $M$ with Clifford multiplication induced by $c$, compatible hermitian metric $h$, admissible connection $\nabla$ and a $\Zz_{2}$-grading $z$ when $\A$ has even rank. Note that by \cite{ALN2} this construction leads to differential operators in the relevant pseudodifferential calculus, i.e. even differential operators generated by the vector fields that are sections of the Lie algebroid $\A$. Now by \cite{LN2001}, there is an associated Dirac operator $D\in\Diff_{\G}(E)$. 

For simplicity, we will denote such a Dirac bundle by  $(E,D)$, the necessary data being understood. 

\subsection*{$K$-tamings}
One can check that, for any $F$-fully elliptic $P\in\Psi_{\G}(E_{0},E_{1})$, we have:
\[
[P]_{F,\ev}:=\pd_{F}([P]_{F}) = \Big[  C^*(\Tau_FM,E_{0}\oplus E_{1}), 
   \begin{pmatrix}
             0 &  \Q|_{\Tau_FM}\\   \P|_{\Tau_FM}  & 0
   \end{pmatrix}\Big] \in KK(\Cc,C^*(\Tau_FM)).
\]
Here $\P$ is a $F\times\{1\}$-fully elliptic lift of $P$, that is  $\P\in \Psi_{\G^{\mathrm{ad}}}$,  $\P\vert_{t=1}=P$ and $\P\vert_{t=0}=\sigma_{\mathrm{pr}}(P)$; while $\Q$ is a similar lift of a full parametrix $Q$ of $P$. We refer to this situation as the \emph{even} case.

Now, if $P\in\Psi_{\G}(E)$ is selfadjoint and $F$-fully elliptic, which we refer to as the  \emph{odd}  case, we will consider instead of $[P]_{F}$ and $[P]_{\pr}$, the classes:
\[
 [P]_{F,\odd} =\Big[  C^*(\Tau_FM,E),  \P|_{\Tau_FM}   \Big] \in KK_{1}(\Cc,C^*(\Tau_FM)),
\]
\[
 [P]_{\pr,\odd} =\Big[  C^*(\A,E),  \sigma_{\pr}(P)   \Big] \in KK_{1}(\Cc,C^*(\A)),
\]
The convention is the same for positive order operators after using a suitable \emph{order reduction} operator $\Lambda$.

\begin{Def}\label{Def:tameable}
We say that  an elliptic operator $P\in\Psi^{*}_{\G}$  is  $F$-tameable if $[P]_{\pr,\ast}$ is in the range of 
\begin{equation}\label{eq:F-forget-map-nct}
  \tau=(\ev_{t=0})_{*} : K_{\ast}(C^{*}(\Tau_{F}M)) \longrightarrow K_{\ast}(C^{*}(\A)).
\end{equation}
\end{Def}
Using the $K$-equivalences $\G^{\F}_{F}\sim_{K} \Tau_{F}M$ and $\G^{\mathrm{ad}}\sim_{K} \A$, we see that the above condition is equivalent to be in the range of
\begin{equation}
  i_{*} : K_{*}(C^{*}(\G^{\F}_{F})) \longrightarrow K_{*}(C^{*}(\G^{\mathrm{ad}})).
\end{equation}
where $i : C^{*}(\G^{\F}_{F})\hookrightarrow C^{*}(\G^{\mathrm{ad}})$ is the inclusion. 

Also, in the even case, the natural homomorphism  $q_{F}: \Sigma_{F} \to C(S^{\ast} \A)$ induces a map: 
\begin{equation}\label{eq:F-forget-map}
q_{F} : K(\mu_{F}) \longrightarrow K(\mu_{pr})
\end{equation}
which coincides with $\tau$ under the suitable isomorphisms. In particular $ q_{F}([P]_{F})=  [P]_{\mathrm{pr}}$ and $F$-tameability of an elliptic operator $P\in \Psi^{*}_{\G}(E_{0},E_{1})$ is equivalent to  $[P]_{\pr}\in \im q_{F}$.

\subsection*{Reduction to Dirac}
We call $(S,D)$ an even Dirac bundle for $\G$ if  $S=S_{+}\oplus S_{-}$ is a $\Zz_{2}$-graded $\mathrm{Cl}(\A)$-module vector bundle over $X$ and $D=\mathrm{Antidiag}(D_{+},D_{+}^*)\in \Psi^1_\G(S)$ is a Dirac operator. Odd Dirac bundles are defined in the same way without gradings.  

\begin{Def}
Any $F$-fully elliptic operator of the form:
 $$ B = (D \oplus D') +R $$
where 
\begin{enumerate}
\item $(S',D')$ is a Dirac bundle of the same parity as $(S,D)$,
 \item $R\in \Psi^{-\infty}_{\G}(S\oplus S')$,
\item   $[B]_{\pr,\ast}= [D]_{\pr,\ast}$,
\end{enumerate} 
is called \emph{$F$-taming} of $D$.
\end{Def}

\begin{Thm}\label{thm:diracification}
Let $(S,D)$ be a Dirac bundle. Then for any $F$-fully elliptic operator $P\in \Psi_{\G}$ such that 
$[P]_{\pr,\ast}= [D]_{\pr,\ast}\in K^{\ast}(\A)$,
there exists a $F$-taming $B$ of $(S,D)$ such that 
\begin{equation}
 [B]_{F,\ast} = [P]_{F,\ast}\in K_{\ast}(C^{*}(\Tau_{F}X))\simeq K(\mu_{F}).
\end{equation}
\label{Thm:26}
\end{Thm}
\begin{proof}[Proof of the theorem] 
We begin with the even case and use for that the relative $K$-theory description. 
For any integer $N$, we can endow the trivial bundles $X\times\Cc^k$ with a $\mathrm{Cl}(\A)$-module structure (by embedding isometrically the bundle  $\A$ into some trivial bundle over $X$), where  $k\ge N $ is chosen in order to get an even Dirac bundle $(X\times\Cc^k,D_k)$. Then  $[D_{k,+}\oplus (-D_{k,+})]_{pr}=0\in K(\mu_{\mathrm{pr}})$ and we set $D'=D_{k}+(-D_{k})$. 
 
Let $P_{+}\in \Psi_{\G}(E_{+},E_{-})$ be a $F$-fully elliptic element such that $[P_{+}]_{pr}= [D_{+}]_{pr}$. By definition of $K(\mu_{\mathrm{pr}})$,  there exist elementary elements   
 \[ 
  (\xi,\xi,\alpha),(\xi',\xi',\alpha')\in \Gamma(\mu_{\mathrm{pr}})
 \]
such that 
 \[
    (\E_{+},\E_{-},\sigma_{\mathrm{pr}}(P_{+}))  + (\xi,\xi,\alpha)   \simeq  (\cS_{+},\cS_{-},\sigma_{\mathrm{pr}}(D_{+})) + (\xi',\xi',\alpha'). 
 \]
 Adding to both sides another elementary element, and renaming, we can assume that $\xi'=C(S(\A^*), \Cc^{k})$. We rename $\cS_{\pm}\oplus C(S(\A^*), \Cc^{k})$ into $\cS_{\pm}$ and identify $\E_{\pm}\oplus \xi\simeq \cS_{\pm}$.  Replacing $\alpha$ by $1$ and $\alpha'$ by $\sigma_{\mathrm{pr}}(D_{+}')$ after homotopies, and renaming $P_+\oplus 1$ into $P_+$, we get:
 \[
    \big(\cS_{+},\cS_{-},\sigma_{\mathrm{pr}}(P_+)\big)  \sim  \big(\cS_{+},\cS_{-},\sigma_{\mathrm{pr}}((D\oplus D')_{+})\big) ,
 \]
 and in particular 
 \begin{equation}\label{eq:homotopy-sigma-Q} 
  \sigma_{\mathrm{pr}}(P_+)^{-1}\sigma_{\mathrm{pr}}((D\oplus D')_{+})\sim \Id \text{ within } \mathrm{Aut}(\cS_{+}).
 \end{equation}
Now we use the homotopy lifting argument of \cite[Proposition 4.3]{DSconnect}.   Since $\cS_{\pm}=p_{\pm} C(S(A^*))^{N}$ for some projectors $p_{\pm}\in M_{N}(C(M))$,  the surjectivity of $q :  \Sigma_{F}\longrightarrow  C(S(\A^*))$ implies the surjectivity of 
  \[
  \widetilde{q}: \L(p_{+}(\Sigma_{F}^{N})) \longrightarrow  \L(\cS_{+}).
  \] 
 By the open mapping theorem, $\widetilde{q}$ is then open. Thus,  $\widetilde{q}(\mathrm{Aut}(p_{\pm}(\Sigma_{F}))_{(0)})$ is an open, and therefore closed, subgroup of $\mathrm{Aut}(\cS_{+})$, therefore: 
 \[ 
  \widetilde{q}(\mathrm{Aut}(p_{\pm}\Sigma_{F}^{N})_{(0)})=\mathrm{Aut}(\cS_{\pm})_{(0)}. 
 \]
By \eqref{eq:homotopy-sigma-Q},  we can choose  $y\in \mathrm{Aut}(p_{+}\Sigma_{F}^{N})_{(0)}$ such that 
 \begin{equation}\label{eq:homotopy-lifting} 
  \sigma_{\mathrm{pr}}(P_+) \widetilde{q}(y)=\sigma_{\mathrm{pr}}(P_+)y=\sigma_{\mathrm{pr}}((D\oplus D')_{+}).
 \end{equation}
Let $Y\in \Psi_{\G}(\cS_{+})$ such that $\sigma_{F}(Y)=y$. Let us set $T=\frac{1}{\sqrt{2}}\begin{pmatrix} 0 & (P_+Y)^{*} \\ P_+Y & 0 \end{pmatrix}$. Since $y$ is in the connected component of the identity in $\mathrm{Aut}(\cS_{\pm})$, we have $[T_{\pm}]_{F}=[P_{\pm}]_{F}$. If  $\Delta\in\Psi^{2}_{\G}(\cS)$ denotes a Laplacian, Equation \eqref{eq:homotopy-lifting} gives:
\begin{equation}
  T =  (1+\Delta)^{-1/2}(D\oplus D') + R \text{ with } R\in C^{*}(\G,\End(\cS)).
\end{equation}
We know that  the class  $\sigma_{F}(T)$  of $T$ in the $C^{*}$-algebra $\Psi_{\G}(\cS)/C^{*}(G_{O},\End(\cS))$ is invertible. By density of $C^{\infty}_{c}(\G,\End(\cS))$ into $C^{*}(\G,\End(\cS))$, we can choose $R_{1}\in C^{\infty}_{c}(\G,\End(\cS))$ such that $\sigma_{F}(T_{1})$  remains invertible, where $T_{1}=T-R+R_{1}$. We write $ R_{1}=(1+\Delta)^{-1/2}(1+\Delta)^{1/2}R_{1}$ and we approximate $(1+\Delta)^{1/2}R_{1}\in C^{*}(\G,\End(\cS))$ by an  element $R_{2}\in C^{\infty}_{c}(\G,\End(\cS))$ close enough so that  $\sigma_{F}(T_{2})$  remains invertible, where 
\[  T_{2} = (1+\Delta)^{-1/2}(D\oplus D') + (1+\Delta)^{-1/2}R_{2} .\]
Thus, $B:=(D\oplus D') + R_{2}$ is a $F$-fully elliptic operator and $R_{2} \in \Psi^{-\infty}_{\G}(\cS)$. Without loss of generality, we can assume that $R,R_{1},R_{2}$ are close enough so that $T$ and $T_{2}$ are homotopic among $F$-fully elliptic operators. We conclude: 
\[
    [B_{+}]_{F}=[T_{+}]_{F}=[P_{+}]_{F}=-[P_{-}]_{F}=- [B_{-}]_{F}.
\]

In the odd case now, we consider the map $\tau_{*}: KK_{1}(\Cc,C^{*}(\Tau_{F}X))\longrightarrow  KK_{1}(\Cc,C^{*}(\A))$.
Let $\P$ be a fully elliptic lift of $P$ and $p\in M_{N}(C(X))$ a projector such that $E=p(X\times\Cc^{N})$.
We have  $\tau_{*}[pC^{*}(\Tau_{F}X)^{N}, \P\vert_{\Tau_{F}X}] =[pC^{*}(\A)^{N},\sigma_{\mathrm{pr}}(P)]$, therefore by assumption $[pC^{*}(\A)^{N},\sigma_{\mathrm{pr}}(P)]=[pC^{*}(\A)^{N},\sigma_{\mathrm{pr}}(D)]\in KK_{1}(\Cc,C^{*}(\A))$. This means that after the addition of degenerate Kasparov modules we have a homotopy between $\sigma_{\mathrm{pr}}(P)$ and $\sigma_{\mathrm{pr}}(D)$ within the automorphisms of the $C^{*}$-algebra $\L(\tau\E)/\K(\tau \E)$ where $\E=C^{*}(\Tau_{F}X,E)$ and $\tau \E = C^{*}(\A,E)=\E\otimes_{\tau}C^{*}(\A)$. 
Now by lifting homotopies through the epimorphism of $C^{*}$-algebras:
\[
   \widetilde{\tau} : \L(\E)/\K( \E)\longrightarrow \L(\tau\E)/\K(\tau \E)
\]
as in the even case, and using again a density argument, we conclude the proof. 
\end{proof}

\subsection*{Lie manifolds}
If $M$ is a compact manifold with embedded corners, it is understood that a decomposition into closed faces is given, which are themselves compact manifolds with embedded corners inheriting the induced decomposition into faces. The set of codimension $k$ faces  of $M$ is denoted  by $\F_{k}M$. 

We will consider  Lie manifold structures on manifolds with embedded corners \cite{ALN}. Any such structure on $M$ is given by an almost injective Lie algebroid $\A$,  the latter being integrable by \cite{Deb} into a unique, up to isomorphism, $s$-connected Lie groupoid $\G$. We call \emph{integrated Lie manifold} any pair $(M,\G)$ such that 
\begin{enumerate}
\item  $\G\rightrightarrows M$ is a Lie groupoid over $M$;
\item $M_{0}=M\setminus \partial M$ is saturated and $\G_{M_{0}}=M_{0}\times M_{0}$,
\item $\G$ is amenable.
\end{enumerate}

Let  $(M,\G)$ be an integrated Lie manifold. We denote by $\Psi_{\V}^*(M; E_0, E_1)$ the algebra of  pseudodifferential operators of Lie type on $M_{0}=M\setminus \partial M$ acting between the sections of vector bundles $E_{j}\to M$   \cite{ALN}. It coincides with the image of $\Psi^{*}_{\G}(E_{0},E_{1})$ by the vector representation $r_{\#}$ when $\G$ is $s$-connected  \cite{ALN}.  By a slight abuse of notation, we continue to set  $\Psi_{\V}^*=r_{\#}\Psi^{*}_{\G}$ in the general case. Equivalently, $\Psi_{\V}^*$ is  isomorphic to the space obtained by restricting elements of   $\Psi^{*}_{\G}$ to the fiber of $\G$ over any arbitrary interior point $x$. The isomorphism comes then from the diffeomorphism $r : \G_{x}\to M_{0}$. 

The $\partial M$-indicial symbols and $\partial M$-joint symbols can be related with their analogues for any $F\in\F_1(M)$ in the obvious way, and we can summarize this in the following commutative diagram in which the outer rectangle is made of fibered products:
 \[
\xymatrix{
\Sigma_{\partial M}^{m}(E) \ar@{->}[d(1.8)]_{\pi_2} \ar@{->}[r(1.5)]^-{\pi_1} & &  C^{\infty}(S(\A^{\ast}), \overline{\pi}^{\ast} \End(E)) \ar[d(1.8)]^{(r_F)_{F \in \F_1(M)}} & & \\
& \ar@{->>}[ul]_{\sigma_{\pr} \oplus I_{\partial M}} \ar@{->>}[dl]_{I_{\partial M}} \Psi_{\G}^m(E) \ar@{->>}[ur]^{\sigma_{\pr}} \ar@{->>}[dr]_{} & \\
\prod_{F\in \F_{1}(M)} \Psi^{m}_{\G_{F}}(E|_{F}) \ar@{->>}[r(1.4)]^-{\oplus_{F} \sigma_{\pr, F}} & &  \prod_{F} C^{\infty}(S(\A_{|F}^{\ast}), \overline{\pi}^{\ast} \End(E_{|F})) & &
}
\]
We observe that with $F=\partial M$, we have $O=M\setminus \partial M=:M_0$ and thus $\G_O=M_0\times M_0$. It follows that the ideal in \eqref{eq:F-joint-symbol-seq} is the ideal of compact operators and that  $\ind_{\partial M}$ takes integer values. 
Since  $r_{\#}: C^*(\G_O,\End(E))\overset{\simeq}{\to} \K(L^2_\V(M,E))$ and $r_{\#} : \Psi_{\G}(E) \hookrightarrow \B(L^2_\V(M,E)) $ \cite{ALN},  the sequence \eqref{eq:F-joint-symbol-seq} also shows that $r_{\#}(P) : H_{\V}^s(M; E) \to H_{\V}^{s-m}(M; E)$ is Fredholm, when $P\in \Psi_{\G}^m(E)$ is \emph{$\partial M$-fully elliptic}, and that   $\ind_{\partial M}$ computes its Fredholm index. Here  $H_{\V}^s(M,E)$ are the Sobolev spaces of sections of $E$ associated with the Lie structure and $ L^2_\V= H_\V^0$. 

\medskip \emph{From now on, whenever we take $F=\partial M$ as a closed saturated subspace of an integrated Lie manifold $(M,\G)$, we will omit it in the notation, for instance: fully means $\partial M$-fully, $\G^\F$ denotes $\G^\F_{\partial M}$, etc ... and we set $\EllV(M)=\Ell_{\partial M}(\G)$.}

 \subsection*{Pullbacks}\label{DeformGrpds}

Let   $\varphi : \Sigma \longrightarrow M$  be a  surjective tame submersion, that is, a surjective submersion satisfying \eqref{eq:tame-submersion}. Then the decomposition into faces $\F_{*}(M)$ can be lifted to a decomposition into faces of $\Sigma$. 
 An interesting example occurs in the so-called clutching spaces and vector bundle modification:

\begin{Ex}
Let $M$ be a manifold with corners and $\pi :V\to M$ be a real vector bundle. Consider the sphere bundle $\varphi: Z=S(V\oplus \Rr)\to M$. Then $\varphi$ is a tame surjective submersion. We say that  $Z$ is the clutching space  of  $\pi :V\to M$. 
\end{Ex}
We recall the notion of pull-backs.
\begin{Def}
Let $(\A, \varrho)$ be a Lie algebroid over $M$ and $\G \rightrightarrows M$ a Lie groupoid over $M$.
Let  $\varphi \colon \Sigma \to M$ be a surjective submersion.

\emph{(1)} The \emph{Lie algebroid pullback} of $(\A, \varrho)$ over $\Sigma$  is the Lie-algebroid $(\Aphi, \rhophi)$ given by:
\[
\Aphi = \{(v, w) \in \varphi^{\ast} \A \times T \Sigma :  \varrho(v)= d\varphi(w) \} \subset \varphi^{\ast} \A \oplus T \Sigma
\text{ and } \rhophi=\mathrm{pr}_2. 
\]

\emph{(2)} The \emph{Lie groupoid pullback} of $\G$ over $\Sigma$ is the Lie groupoid $\Gphi \rightrightarrows \Sigma$ given by 
\[
\Gphi = \{(x, \gamma, y) \in \Sigma \times \G \times \Sigma : \varphi(x) = r(\gamma), \ \varphi(y) = s(\gamma)\},
\]
with structure maps: $\rangephi(x, \gamma, y) = x$, $\sourcephi(x, \gamma, y) = y$, $(x,\gamma, y)(y,\gamma',z)=(x,\gamma\gamma',z)$. 
\label{Def:thickpb}
Note that the Lie groupoid pullback of the Lie algebroid $\A$ also makes sense and will be denoted by $\presuper{\varphi}{\textbf{A}}$ if a risk of confusion occurs. 
\end{Def}


Both notions are consistent since: 
\begin{Thm}
Let $(M, \G)$ be an integrated Lie manifold and  $\varphi \colon \Sigma \to M$ be a tame surjective submersion. 
Then $(\Sigma, \Gphi)$ is an integrated Lie manifold such that  $\A(\Gphi) \cong \Aphi$. \\
Moreover, $\Gphi$ is Morita equivalent to $\G$ and the (\emph{$\V$-related}) Lie structure $\Vphi$ is given by:
\[
\Vphi=\{ \widetilde{V} \in \Gamma(T \Sigma)\ ;  \exists V\in\V, \ d\varphi(\widetilde{V})=V\circ\varphi\}.
\]
 
\label{Thm:thickpb}
\end{Thm}
This is a compilation of known facts: the  isomorphism $\A(\Gphi) \cong \Aphi$ is proved in \cite{Mackenzie}, the Morita equivalence is proved in \cite{DS}, which ensures that $\Gphi$ is amenable too \cite{AR}, that $\Gphi$ gives a Lie structure on $\Sigma$ is done in \cite{N2015} and further developments can be found in \cite{VZ2016}.

\bigskip Let $(M,\G)$ be an integrated Lie manifold and $\varphi \colon \Sigma \to M$ a tame surjective submersion. The operations consisting of taking the adiabatic deformation and taking a pullback do not commute, however there are natural deformation groupoids relating them and this fact is exploited in \cite{VZ2016} in order to obtain pushforward maps (which are there occurrences of wrong way maps). We need to recall the results of \cite{VZ2016} in our notation. Firstly, there is a deformation Lie groupoid:
\begin{equation}\label{eq:thomlike}
 \L_\varphi= (\Gphi)^{\mathrm{ad}} \times\{u=0\} \cup \presuper{\varphi_1}(\G^{\mathrm{ad}}) \times(0,1]_u  \rightrightarrows \Sigma\times[0,1]_t\times[0,1]_u,
\end{equation}
where $\varphi_1=\varphi\times\Id : \Sigma\times[0,1]_t\to M\times[0,1]_t$. 
Let $z$ be  the lift by $\varphi$ of a  defining function for the boundary of $M$ and consider the following saturated subgroupoids of  $\L_\varphi$:
\begin{equation}\label{eq:thomlike-Fredholm}
 \L_\varphi^\F := \L_\varphi|_{(z,t)\not=(0,1)} = (\Gphi)^\F\times\{u=0\} \cup \presuper{\varphi_1}(\G^{\F}) \times(0,1]_u,
\end{equation}
\begin{equation}\label{eq:thomlike-nc}
 \L_\varphi^{\mathrm{nc}}:=\L_\varphi|_{tz=0}= \Tau \Sigma \times\{u=0\} \cup \presuper{\varphi_1}\Tau M \times(0,1]_u  \rightrightarrows \Sigma_{\partial\Sigma}\times[0,1]_u,
\end{equation}
\begin{equation}\label{eq:thomlike-A}
 \L_\varphi^{0}:=\L_\varphi|_{t=0}= \Aphi \times\{u=0\} \cup \presuper{\varphi}{\textbf A}\times(0,1]_u  \rightrightarrows M\times[0,1]_u. 
\end{equation}
 Using the restriction homomorphisms  $\ev_{u=i}$,  $i=0,1$, and the  natural Morita equivalence $\M$ we get homomorphisms (see \cite{VZ2016} for more details):
\begin{equation}\label{eq:push-forward-ad}
  \varphi^{\mathrm{ad}}_{!}=  \M\circ [\ev_{u=1}]\circ [\ev_{u=0}]^{-1} : K_\ast(C^*((\Gphi)^{\mathrm{ad}})) \longrightarrow K_\ast(C^*(\G^{\mathrm{ad}}))).
\end{equation}
Considering the variations \eqref{eq:thomlike-Fredholm}, \eqref{eq:thomlike-nc} and \eqref{eq:thomlike-A} of \eqref{eq:thomlike} leads to similar homomorphisms:
 \begin{equation}\label{eq:push-forward-F}
  \varphi^{\F}_{!} : K_\ast(C^*((\Gphi)^{\F})) \longrightarrow K_\ast(C^*(\G^{\F}))),
\end{equation}
\begin{equation}\label{eq:push-forward-nct}
  \varphi^{\mathrm{nc}}_{!} : K_\ast(C^*(\Tau \Sigma)) \longrightarrow K_\ast(C^*(\Tau M)),
\end{equation}
\begin{equation}\label{eq:push-forward-A}
  \varphi^{0}_{!} : K_\ast(C^*(\Aphi)) \longrightarrow K_\ast(C^*(\A)).
\end{equation}
 Then the following result is essentially a rephrasing of  \cite{VZ2016}:
 \begin{Thm}\label{Thm:pushforward-respect-indices}
Let $(M, \G)$ be an integrated Lie manifold and $\varphi : \Sigma \to M$ a tame surjective submersion. Then the map $  \varphi^{\mathrm{nc}}_{!} $ commutes with  Fredholm index: 
\begin{equation}\label{eq:pushforward-respect-indices}
    \ind_{\partial M}^{\F}\circ  \varphi^{\mathrm{nc}}_{!} = \ind_{\partial \Sigma}^{\F}
\end{equation}
Here the target group of the index maps $ \ind_{\bullet}^{\F}$ is replaced by $\Zz$ after applying the obvious Morita equivalences. 

 In other words, if  $B\in \Psi_{\Gphi}^{*}(\widetilde{E}_{+},\widetilde{E}_{-})$ and $P\in \Psi_{\G}^{*}(E_{+},E_{-})$ are fully elliptic operators  and satisfy $ \varphi^{\mathrm{nc}}_{!} [B]_{\partial\Sigma,\ev}=  [P]_{\partial\Sigma,\ev}$ then $B$ and $P$ have the same Fredholm index. 
 \end{Thm}
\begin{proof}
Using the groupoid $\L_{\varphi}$ and the definition of $\ind^{\F}$ we get a commutative diagram:
\begin{equation}\label{cd:push-comm-ind}
\begin{tikzcd}
   K_0(C^*(\Tau \Sigma))  \arrow[bend right=70]{ddd}{\varphi^{\nc}_{!}}   \arrow[bend left=20]{rrr}{\ind^{\F}_{\partial \Sigma}} &  K_0(C^*((\Gphi)_{\partial\Sigma}^{\F}))  \arrow[l, "\ev_{tz=0}"'] \arrow{r}{\mathrm{ev}_{t=1}} & K_0(C^*(\Sigma_{0}\times\Sigma_{0})) \arrow[r, "\M"] & \Zz \ar[equal]{d} \\
   K_0(C^*(\L^{\nc}_{\varphi}))\arrow[u, " \ev_{u=0}"'] \arrow[d, "\ev_{u=1}"]  &  K_0(C^*(\L^{\F}_{\varphi}))   \arrow[l, "\ev_{tz=0}"'] \arrow[u, "\ev_{u=0}"'] \arrow[d, "\ev_{u=1}"] \arrow{r}{\mathrm{ev}_{t=1}} & K_0(C^*(\Sigma_{0}\times\Sigma_{0}\times [0,1])) \arrow[u, "\ev_{u=0}"']\arrow[d, "\ev_{u=1}"]  \arrow[r, "\M\times p_{*}"] & \Zz \ar[equal]{d}  \\
   K_0(C^*(\presuper{\varphi}\Tau M))\arrow[d, "\M"]  &  K_0(C^*(\presuper{\varphi}(\G_{\partial M}^{\F})))   \arrow[l, "\ev_{tz=0}"'] \arrow[d, "\M"]  \arrow{r}{\mathrm{ev}_{t=1}} & K_0(C^*(\Sigma_{0}\times\Sigma_{0}))\arrow[d, "\M"]   \arrow[r, "\M"] & \Zz \ar[equal]{d} \\
 K_0(C^*(\Tau M))  \arrow[bend right=20]{rrr}{\ind^{\F}_{\partial M}} &  K_0(C^*(\G_{\partial M}^{\F}))   \arrow[l, "\ev_{tx=0}"'] \arrow{r}{\mathrm{ev}_{t=1}} & K_0(C^*(M_{0}\times M_{0}))  \arrow[r, "\M"] & \Zz \\ \end{tikzcd}
\end{equation}
Above $p=\ev_{u=\alpha}$ for arbitrary $\alpha\in [0,1]$. The result follows immediately. 
\end{proof}

\begin{Thm}\label{Thm:Thom-Morita-Kfunctor}
Let $(M, \G)$ be an integrated Lie manifold,  $\pi : V\to M$ be a real vector bundle and denote by $\varphi : \Sigma=S(V\oplus\Rr)\to M$  the associated clutching and $i : V\hookrightarrow \Sigma$ the embedding as an open hemisphere. 
Then 
\begin{enumerate}
 \item $\pi_!^{nc}$ and $\pi_!^{0}$ are isomorphisms, the latter being the inverse of the Thom isomorphism of the complex bundle $\presuper{\pi}\A\to A$. 
 \item We have the identities:
 \begin{equation}\label{eq:Thom-pushforward}
 \varphi_!^{\bullet} \circ \big( i_*  \circ (\pi_!^{\bullet})^{-1}) = \Id ,  \quad \text{ with } \bullet =nc  \text{ or }  0. 
 \end{equation}
\end{enumerate}
\end{Thm}

\begin{proof}
Considering the groupoids $\L_\pi$, $\L_\varphi$ and the embeddings provided by $i$, we get  commutative diagrams:
    
\begin{equation}\label{cd:thomlike}
\begin{tikzcd}
 \cdots \arrow{r} & K_0(C^*(\Gphi|_{\partial\Sigma}\times (0,1)) \arrow{r}\arrow[d, "\M"]  &  K_0(C^*(\Tau \Sigma)) \arrow{r}{\mathrm{ev}_{t=0}}\arrow[d, "\varphi^{\mathrm{nc}}_{!}"]  & K_0(C^*(\Aphi)) \arrow[d, "\varphi^{0}_{!}"] \arrow{r} & \cdots \\
 \cdots \arrow{r} & K_0(C^*(\G|_{\partial M}\times (0,1))) \arrow{r}&  K_0(C^*(\Tau M)) \arrow{r}{\mathrm{ev}_{t=0}}&  K_0(C^*(\A))  \arrow{r} & \cdots
\end{tikzcd}
\end{equation}
\begin{equation}\label{cd:thomlikepi}
\begin{tikzcd}
 \cdots \arrow{r} & K_0(C^*(\presuper{\pi}\G |_{\partial V}\times (0,1)) \arrow{r}\arrow[d, "\M"]  &  K_0(C^*(\Tau V)) \arrow{r}{\mathrm{ev}_{t=0}}\arrow[d, "\pi^{nc}_{!}"]  & K_0(C^*(\presuper{\pi}\A)) \arrow[d, "\pi^{0}_{!}"] \arrow{r} & \cdots \\
 \cdots \arrow{r} & K_0(C^*(\G|_{\partial M}\times (0,1))) \arrow{r}&  K_0(C^*(\Tau M)) \arrow{r}{\mathrm{ev}_{t=0}}&  K_0(C^*(\A))  \arrow{r} & \cdots
\end{tikzcd}
\end{equation}
\begin{equation}\label{cd:thomlike3}
 \xymatrix{
 K_0(C^*(\Tau \Sigma))   \ar[r]^{\varphi_!^{nc}} &K_0(C^*(\Tau M))  \\
 \K_0(C^*(\Tau V)) \ar[u]_{i_*}\ar[r]^{\pi_!^0} & K_0(C^*(\Tau M))   \ar[u]^{\Id} 
}
\end{equation}
\begin{equation}\label{cd:thomlike2}
 \xymatrix{
 K_0(C^*(\Aphi))   \ar[r]^{\varphi_!^{0}} &K_0(C^*(\A))  \\
 \K_0(C^*(\presuper{\pi}\A)) \ar[u]_{i_*}\ar[r]^{\pi_!^0} & K_0(C^*(\A))   \ar[u]^{\Id} 
}
\end{equation}
Since  $\L_{\pi}^{0}$ is the Thom groupoid of the complex bundle  $\presuper{\pi}\A\to A $, we get that $\pi^{0}_{!}$ is the inverse of the Thom isomorphism of this complex bundle  \cite{DLN}. By Diagram \eqref{cd:thomlikepi} and the five lemma, $\pi_!^{nc}$ is an isomorphism too. Now  Diagrams  \eqref{cd:thomlike2} and  \eqref{cd:thomlike3} give: 
\begin{equation}
 \varphi_!^{\bullet} \circ \big( i_*  \circ (\pi_!^{\bullet})^{-1}) = \Id ,  \quad \text{ with } \bullet =nc  \text{ or }  0.  
\end{equation}
This proves that  $\varphi_!^{0}$ and  $\varphi_!^{nc}$ are surjective and gives explicit sections. 
%
\end{proof}

Finally, we denote by $\varphi_* : K(\mu_{\partial \Sigma})\longrightarrow K(\mu_{\partial M})$ the map induced by  $ \varphi^{\mathrm{nc}}_{!} $, through  the isomorphism of Theorem \ref{Thm:Fh-F}. Note that the homomorphisms $q_{\partial\Sigma}$ and $q_{\partial M}$ are then exchanged with the homomorphisms $\mathrm{ev}_{t=0}$.

\section{Geometric $K$-homology}
\label{Kgeo}

\subsection*{Geometric cycles}
Let $(M,\G)$ be an integrated Lie manifold.  

\begin{Def}\label{Def:geocycles}
 An even (odd) geometric cycle over $(M,\G)$ is a 4-tuple $x=(\Sigma, \varphi, E, B)$ consisting of 
\begin{enumerate}[(1)]
 \item an even (odd) dimensional compact manifold with corners $M$; 
 \item a tame surjective submersion  $\varphi: \Sigma\longrightarrow M$;  
 \item an even (odd) Dirac bundle $(E,D)$ on  $(\Sigma,\Gphi)$;
\item a self-adjoint even (odd) Dirac $\partial \Sigma$-taming  $B$ of $D$.
\end{enumerate}
The set of geometric cycles of parity $j$ is denoted by $\EVgeo_j(M)$. 
\end{Def}
 
 If the geometric cycle $x=(\Sigma, \varphi, E, B)$ is even, we get a class 
 $$[B]_{0}:=[B_{+}]_{\partial \Sigma,\ev}\in  K_0(C^*(\Tau \Sigma))$$
  and if it is odd, we get a class 
  $$[B]_{1}:= [B]_{\partial \Sigma,\odd}\in KK_{1}(\Cc,C^*(\Tau \Sigma))\simeq K_1(C^*(\Tau \Sigma)).$$
  Applying the pushforward map $\varphi_{!}^{\nc}$, we get classes in $K_\ast(C^*(\Tau M))$ as well.
  \begin{Def}\label{Def:iso-geocycles} Let $x=(\Sigma, \varphi, E, B)$ be a geometric cycle of parity $j$.  A geometric cycle $x'=(\Sigma', \varphi', E', B')$ is isomorphic to $x$ if there exists a
 diffeomorphism $\kappa :\Sigma\to \Sigma'$ such  that 
 \begin{enumerate}
  \item $\varphi'\circ \kappa=\varphi$ and  $(E,D)\simeq \kappa^* (E',D')$ is an isomorphism of  Dirac bundles.
  \item $[B]_{j}=\kappa^*[B']_{j}\in K_{j}(C^{*}(\Tau\Sigma))$.
 \end{enumerate}  
 \end{Def}

\begin{Def}\label{Def:union-opposite-geocycle}\ 
 \begin{enumerate}
  \item The disjoint union of $x_i=(\Sigma_i, \varphi_i, E_i, B_i)\in \EVgeo_*(M)$, $i=0,1$, is  
  \begin{equation}
    x_0\cup x_1 = \left( \Sigma_0 \cup \Sigma_1, \varphi_0 \cup \varphi_1, E_0 \cup E_1, B_0\cup B_1\right)\in \EVgeo_*(M).
  \end{equation}
  \item The direct sum of $x_i=(\Sigma, \varphi, E_i, B_i)\in \EVgeo_*(X)$, $i=0,1$, is  
  \begin{equation}
    x_0\oplus x_1 = \left( \Sigma, \varphi, E_0 \oplus E_1, B_0\oplus B_1\right)\in \EVgeo_*(M).
  \end{equation}
 \item If $x=(\Sigma, \varphi, E, B)\in \EVgeo_i(M)$ then the opposite cycle is
 \begin{equation}
   -x=(\Sigma, \varphi, E^{op}, B^{op}).
 \end{equation}
Here $E^{op}$ is the same bundle, with the opposite grading in the even case, and $B^{op}=-B$. We have $[B^{op}]_{j}=-[B]_{j}$  for $j=0,1$.
   \end{enumerate}
\end{Def}

\subsection*{Cobordisms}
 To define cobordism of geometric cycles, we introduce:
\begin{Def}\label{Def:cobordism-over-mwc}
A  cobordism  over $M$ is a triple  $(W,Y,\Phi)$ such that 
\begin{enumerate}
 \item $\Phi: W\longrightarrow M$ is a surjective submersion,
 \item $Y=H_1\cup\cdots\cup H_k$ is a union of boundary faces  of $W$ called relative boundary faces, the  remaining boundary faces being called absolute and absolute faces are assumed to be pairwise disjoint,
 \item $\Phi^{-1}(\partial M)=Y$ and  for all relative boundary faces $H$,  $\Phi(H)\in \F_1(M)$,
 \item If  $H\in\F_1(W)$  is absolute, then $\Phi|_H : H\to M$ is tame.
\end{enumerate}
\end{Def}
The submersion $\Phi: W\to M$ is no longer tame. As a consequence, if $(M,\G)$ is an integrated Lie manifold,  the absolute faces  of $W$ are not saturated for the pullback groupoid $\presuper{\Phi}\G$. We will replace the latter by \cite{M,DS}:
\begin{equation}
  \presuper{\Phi_b}\G \rightrightarrows W
\end{equation}
defined by blowing up successively all the absolute faces $H_0,\ldots, H_\ell$ of $W$. Firstly, set:
\begin{equation}
 \G_0=\presuper{\Phi}\G, \qquad \beta_0=\Id,\qquad \H_0=\presuper{\Phi|_{H_{0}}}\G
\end{equation}
and then for any $i\ge 0$: 
\begin{equation}
 \G_{i+1}=\mathrm{\Sblup}_{r,s}(\G_i;\H_i) \overset{\beta_{i+1}}{\longrightarrow} \G_i, \quad \H_{i+1}=(\beta_1\circ\ldots\beta_{i+1})^{-1}(
 \presuper{\Phi|_{H_{i+1}}}\G(M) )
\end{equation}
The required groupoid is $ \presuper{\Phi_b}\G =\G_\ell$ and it does not depend on the order of blowups for the initial submanifolds $ \presuper{\Phi|_{H_i}}\G(M)$ are pairwise disjoint. This is true for any $i$ for the blowup manifolds $\mathrm{\Sblup}(\G_i;\H_i)=[\G_{i},\H_{i}]$ by \cite{MelroseMWC}, and since the structural maps coincide over the open dense subset $(W\setminus \partial W)^{2}$, the result follows. We now have by construction: 
\begin{enumerate}
\item  all faces $H\in\F_1(W)$ are saturated for $\presuper{\Phi_b}\G$.
 \item For any absolute $H\in\F_1(W)$, we consider an open neighborhood $H\subset \U\simeq H\times [0,+\infty)$ and then we have by  \cite[Paragraph 5.3.5]{DS} (after rearranging if necessary the order of blowups in order to have $H$ at the end):
 \begin{align}\label{eq:collar-neigh-bordism}
 (\presuper{\Phi_b}\G)_{\U}& \simeq \big(\mathrm{\Sblup}_{r,s}(W^{2},H^{2})\times_{W^{2}} \G_{\ell-1}\big)\vert_{U} \nonumber \\
  & \simeq (\mathrm{\Sblup}_{r,s}(W^{2},H^{2})\vert_{\U} \, \times_{\U^{2}}\,  \presuper{\Phi\vert_\U}(\G) \nonumber\\
  & \simeq  (H^{2}\times [0,+\infty)\rtimes \Rr) \, \times_{\U^{2}}\,  \presuper{\Phi\vert_\U}(\G), \text{ using  }   \U\simeq H\times [0,+\infty).
 \end{align}
\end{enumerate}
In particular 
\begin{equation}\label{eq:collar-bordism}
 (\presuper{\Phi_b}\G)_{H} \simeq  \Rr \times \presuper{\Phi\vert_H}(\G)
\end{equation}
The last isomorphism will be called a boundary decomposition of $\presuper{\Phi_b}\G$ and the choice of collar diffeomorphism $ \U\simeq H\times [0,+\infty)$  (coming from the choice of a defining function)   for the absolute faces will be considered as part of the data in the sequel. The identifications above provide an isomorphism 
\[
  (\presuper{\Phi_b}\A)_{\U} \simeq   \rho^{*}(\Rr \times \presuper{\Phi\vert_H}(\A)),
\]
where $\rho :\U\to H$ corresponds to the first projection through the collar diffeomorphism.  A Dirac bundle $(E,D)$  over $(W,\presuper{\Phi_b}\G)$ is locally of product type, if it satisfies \cite[Definition 3.4]{B2009}, with $TH\times TN$ replaced here by $T[0,1) \times \presuper{\Phi\vert_H}(\A)$ and using the isomorphism above. In such a case, we can apply the boundary reduction of \cite[Paragraph 3.2]{B2009}, replacing  $H$ and $TN$ there respectively by $[0,+\infty)$ and $ \presuper{\Phi\vert_H}(\A)$ here, and we get a Dirac bundle (or rather an isomorphism class of Dirac bundles) over $(H, \presuper{\Phi\vert_H}(\G))$ of the opposite parity denoted by $(E_{H},D_{H})$ and called the boundary reduction of $(E,D)$ to the absolute face $H$.

\begin{Def}\label{Def:bordism}\ 
A  (even, odd) cobordism over $(M,\G)$ is a 5-tuple $w=(W, Y, \Phi, E, B)$ where:
\begin{enumerate}
 \item $(W,Y,\Phi)$ is a (even, odd dimensional) manifold cobordism  over $M$;
 \item A (even, odd) Dirac bundle $(E,D)$ on $(W,\presuper{\Phi_b}\G)$ is given; 
 \item a self-adjoint (even, odd)  Dirac $Y$-taming $B$ of $D$.
\end{enumerate}
A null-cobordism is a cobordism with exactly one absolute face. 
\end{Def}
If $j$ is the parity of the cobordism, we get a class $[B]_{j}\in K_j(C^*(\presuper{\Phi_b}\G^\F_Y)))\simeq K_{j}(C^{*}(\Tau_{Y}W))$.
 
Let  $w=(W, Y, \Phi, E, B)$ be an odd cobordism over  $(M,\G)$, let $\varphi :\Sigma \to M$ where $\Sigma$ is an absolute face of $W$. Denote by $D$ the underlying Dirac operator of $w$ and pick up a  boundary reduction $(E_{\Sigma},D_{\Sigma})$ of $(E,D)$. We consider the commutative diagram:
\begin{equation}\label{cd:cobord-2-boundary}
\begin{tikzcd}
  K_1(C^*(\Tau_{Y}W)) \arrow{r}{\rho_{\Sigma}}\arrow{d}{\ev_{t=0}} & K_1(C^*(\Rr\times \Tau \Sigma)) \arrow{r}{\text{Bott}}\arrow{d}{\ev_{t=0}}  & K_0(C^*(\Tau \Sigma)) \arrow{d}{\ev_{t=0}}  \\
 K^1(\presuper{\Phi_{b}}\A) \arrow{r}{\rho_{\Sigma}}&    K^{1}(\Rr\times \Aphi) \arrow{r}{\text{Bott}}&  K^{0}_{c}(\Aphi) 
\end{tikzcd}
\end{equation}
where $\rho_{\Sigma}$ is the  composition of the natural restriction homomorphism 
\[
  C^{*}(\Tau_{Y}W) \longrightarrow C^{*}( (\presuper{\Phi_{b}}\G)^{\F}\vert_{\Sigma})
\]
with the isomorphism  induced by \eqref{eq:collar-bordism} at the level of adiabatic and Fredholm groupoids: 
\[
 C^{*}( (\presuper{\Phi_{b}}\G)^{\F}\vert_{\Sigma})\simeq  C^{*}( \Rr\times (\presuper{\varphi}\G)^{\F})
\]
and then with the morphism $\ev_{zt=0} : C^{*}(\presuper{\varphi}\G)^{\F})\to C^{*}(\Tau \Sigma)$ (that is, the one which  gives   the $K$-equivalence
$
C^{*}( \Rr\times (\presuper{\varphi}\G)^{\F})\overset{K}{\sim} C^{*}( \Rr\times \Tau \Sigma)
$).
It follows immediately from the diagram \eqref{cd:cobord-2-boundary} that 
\begin{equation}
 \ev_{t=0}\big(\mathrm{Bott}\circ\ev_{zt=0}\circ \rho_{\Sigma}([B]_{Y,\odd})\big) = [D_{\Sigma}]_{\pr,\ev}.
\end{equation}
 In particular, $D_{\Sigma}$ is $\partial \Sigma$-tameable and we can pick up a taming $B_{\Sigma}$ using Theorem \ref{Thm:26} such that:
  \begin{equation}\label{eq:cobord-2-boundary-unicity}
     [B_{\Sigma}]_{\partial\Sigma,\ev}=\mathrm{Bott}\circ\rho_{\Sigma}([B]_{Y,\odd}). 
   \end{equation}
 We get an even geometric cycle
 \begin{equation}\label{eq:cobord-2-boundary-def}
     \big(\Sigma, \varphi, E_{\Sigma}, B_{\Sigma} \big) \in  \EVgeo_{0}(M).
 \end{equation}
 Picking up a different representative $(E'_{\Sigma},D'_{\Sigma})$ and a different taming $B'_{\Sigma}$ satisfying \eqref{eq:cobord-2-boundary-unicity} provides isomorphic geometric cycles.
 \begin{Def}\label{Def:cd:cobord-2-boundary}
 With the notation above, the isomorphism class of \eqref{eq:cobord-2-boundary-def} is called the boundary reduction of $w$ to $\Sigma$. 
 \end{Def}
The case of the other parity  is similar.

\begin{Def}\label{Def:bordism-geocycle}\ \\
$\bullet$  Two geometric cycles $x_i=(\Sigma_i, \varphi_i, E_i, B_i)$ over $(M,\G)$, $i=0,1$,  are cobordant if there exists a cobordism  $w=(W, Y, \Phi, E, B)$ such that
\begin{enumerate}
 \item The absolute faces of $W$ are exactly $\Sigma_0$ and $\Sigma_1$;
 \item $\Phi|_{\Sigma_i}=\varphi_i$ 
 \item The boundary reduction of $w$ with respect to $\Sigma_{i}$ is the isomorphism class of $(-1)^{i}x_{i}$. 
 \end{enumerate}

$\bullet$ A geometric cycle $x=(\Sigma, \varphi, E, \Psi)$ over $(M,\G)$ is null-cobordant if there exists a null-cobordism with $\Sigma$ as unique absolute face. 
\end{Def}

\begin{Ex}\label{Ex:easy-cobordisms}\ 
 \begin{enumerate}
  \item  Let $x_{i}=(\Sigma, \varphi, E, B_{i})\in \EVgeo_{\ast}(M)$, $i=0,1$, be two even cycles connected by a \emph{tame homotopy}. This means 
  that there exists a $C^{\infty}$ homotopy of Dirac bundles  $(E,D_{t})$ (i.e., $C^{\infty}$ homotopies of  Clifford homomorphisms and of connections) and a family $(B_{t})_{t\in[0,1]}$ of tamings of $(D_{t})_{t\in[0,1]}$  connecting $B_{0}$ to $B_{1}$ and such that:
 \[
 \sigma_{\partial\Sigma}(\Lambda^{-1}B_{t})\in \Sigma_{\partial \Sigma}(E)^{\times}
 \]
 is a continuous path. In particular: $[B_{0,+}]_{\partial \Sigma,\ev}=[B_{1,+}]_{\partial \Sigma,\ev}\in \K_{0}(C^{*}(\Tau\Sigma))$.
 The cycles $x_{0}$ and $x_{1}$  are then  cobordant. A cobordism  is given by 
  \[
    w=\left(\Sigma\times[0,1]_t,  Y,\varphi', E', B'\right)
  \]
  where:
  \begin{enumerate}
 \item $Y= \partial\Sigma \times[0,1]$, $\varphi'=\pr_{1}\circ\varphi$ and  $E'=\pr_{1}^{*}E\otimes \Cc$.
  \item  $(E',D')$ is  the product of $(E,D_{t})$, with $([0,1]\times \Cc, it(1-t)\frac{\partial}{\partial t})$. Thus: 
   \begin{equation}
     D' = D_{t} + it(1-t)\frac{\partial}{\partial t}.
  \end{equation}
    \item The absolute faces are $\Sigma_0=\Sigma\times\{0\}$ and $\Sigma_1=\Sigma\times\{1\}$. 
    \item We then consider the commutative diagram:
    \begin{equation}\label{cd:taming-self-corbordism}
 \xymatrix{
  K_1(C^*(\Tau_{Y} (\Sigma\times [0,1])))   \ar[r]^{\tau}\ar[d]^{\ev_{t=0}} &K^{1}(\Aphi\times \presuper{b}T[0,1])\ar[d]^{\ev_{t=0}}   \\
 K_1(C^*((\Tau\Sigma)\times\Rr))   \ar[r]^{\tau} &K^{1}(\Aphi\times\Rr)   \\
 K_0(C^*(\Tau\Sigma))\ar[u]^{\beta}    \ar[r]^{\tau} &K^{0}(\Aphi)\ar[u]^{\beta} 
}
\end{equation}
The vertical arrows are isomorphisms. Since $\ev_{t=0}[D']_{\pr,\odd}= [D_{0,+}]_{\pr,\ev}\otimes \beta$ and $ [D_{0,+}]_{\pr,\ev}$ belongs to $\im \tau_{*}$, we get that $[D']_{\pr,\odd}$ belongs to $\im \tau_{*}$ too and we apply Theorem \ref{Thm:26} to choose a $Y$-taming $B'$ of $D'$ such that  $\ev_{t=0} [B']_{Y,\odd}=   [B_{0,+}]_{\partial \Sigma,\ev}\otimes \beta$. 
\end{enumerate}
 
  \item For any $x\in \EVgeo(M)$, the geometric cycle $x\cup(-x)$ is null-cobordant. A null-cobordism  is given by 
  the previous cobordism in which $\Sigma\times\{0,1\}$ is considered as the unique absolute face. 

  \item Let $x=(\Sigma, \varphi, E, B)\in \EVgeo_0(M)$ with underlying Dirac operator denoted by $D$.
Let $(E',D')$ be the product of $(E,D)$ with the spin Dirac bundle $(E_{2},D_{2})$ of $ \mathbb{D}^2$.  Let $(E_{1},D_{1})$ be the boundary reduction of $(E_{2},D_{2})$, that is the spin Dirac bundle associated with the spin structure of $ \mathbb{S}^1$ that bounds the one of $(E_{2},D_{2})$ of $ \mathbb{D}^2$. Let also $(E'',D'')$ be the product of $(E,D)$ with $(E_{1},D_{1})$. 
Consider the commutative diagram:
\begin{equation}\label{cd:taming-null-corbordism}
 \xymatrix{
  K_0(C^*(\Tau_{(\partial \Sigma)\times\mathbb{D}^2} (\Sigma\times\mathbb{D}^2)))   \ar[r]^-{\tau}\ar[d]^{\rho} &K^{0}(\Aphi\times \presuper{b}T\mathbb{D}^2)\ar[d]^{\rho}   \\
 K_0(C^*(\Tau_{\partial \Sigma\times\mathbb{S}^1} (\Sigma\times\mathbb{S}^1)\times\Rr))   \ar[r]^-{\tau} &K^{0}(\Aphi\times T\mathbb{S}^1\times\Rr)   \\
 K_1(C^*(\Tau_{\partial \Sigma\times\mathbb{S}^1} (\Sigma\times\mathbb{S}^1)))  \ar[u]^{S}    \ar[r]^-{\tau} &K^{1}(\Aphi\times T\mathbb{S}^1)\ar[u]^{S} \\
K_0(C^*(\Tau_{\partial \Sigma} (\Sigma)))  \ar[u]^{\otimes D_{1}}    \ar[r]^{\tau} &K^{0}(\Aphi)\ar[u]^{\otimes D_{1}} 
}
\end{equation}
The map $\rho$ corresponds to the restriction to the boundary of $\mathbb{D}^2$ and $S$ is the suspension isomorphism. All the vertical arrows are isomorphisms.
Using the lower part of the diagram we first see that there exists a $\partial \Sigma\times \mathbb{S}^1$-taming  $B''$ of $D''$ such that 
 \[
    [B'']_{\partial \Sigma\times\mathbb{S}^1,\odd} = [B]_{\partial \Sigma,\ev}\otimes [D_{1}]_{\odd}.
 \]
  We then set: 
  \[ 
   0_x=\left(\Sigma\times \mathbb{S}^1, \mathrm{pr}_1\circ\varphi, E'', B''\right)\in \EVgeo_{1}(M). 
  \]
  Now using the upper part of the diagram, we obtain a $\partial \Sigma\times \mathbb{D}^2$-taming $B'$ of $D'$ such that 
  \[
    \rho[B']_{(\partial \Sigma)\times\mathbb{D}^2,\ev} = S[B'']_{\partial \Sigma\times\mathbb{S}^1,\odd}.
  \]
  This provides a null-cobordism for $0_{x}$ as follows:
  \begin{equation}
    w=\left(\Sigma\times \mathbb{D}^2,  (\partial\Sigma) \times \mathbb{D}^2,\mathrm{pr}_1\circ\varphi,E', B'\right).
  \end{equation}
  Starting with an odd $x$, we get a null-cobordant cycle $0_{x}$ in the same way. 
 \item Combining the previous constructions, we see that for any $x,y\in \EVgeo(M)$, the geometric cycles $y$ and $y\cup 0_x$ are cobordant.  
\end{enumerate}
\end{Ex}

\subsection*{Vector bundle modification}

Let  $x=(\Sigma, \varphi, E, B)$  be an even geometric cycle over $(M,\G)$ and  $\pi: V\to \Sigma$ an even rank $\Spinc$-vector bundle. Consider the sphere bundle  $Z= S(V \oplus 1_{\Rr})$  of $ \mathrm{pr}_1\circ \pi : V\oplus \Rr\to \Sigma$. The total space $Z$ is a compact manifold with corners and the projection  $\pi_m : Z\to \Sigma$  is  tame, as well as $\varphi_m= \varphi\circ \pi_m : Z \longrightarrow M$. 
Note that 
\begin{align}
  \presuper{\varphi_m}(\G)= \presuper{\pi_m}(\presuper{\varphi}\G) \quad\text{ and } \quad  \presuper{\varphi_m}\A  \simeq \pi_m^*(\presuper{\varphi}\A) \oplus VZ 
\end{align}
where $VZ=\ker d\pi_m$. The $\Spinc$-structure of $V$ induces a $\Spinc$-structure on the bundle $VZ\to \Sigma$ and therefore on the bundle 
\begin{equation}
   \A_{Z}:=\presuper{\varphi_m}\A  \longrightarrow \presuper{\varphi}\A=: \A_{\Sigma}.
\end{equation}
Using $\L_{\pi_m}$, we get a commutative diagram:
\begin{equation}\label{cd:thom}
\begin{tikzcd}
   K_0(C^*(\G(Z)^{\mathrm{ad}})) \arrow{r}{\mathrm{ev}_{t=0}}\arrow[d, "(\pi_m)^{\mathrm{ad}}_!"]  & K_0(C^*(\A_Z)) \arrow[d, "(\pi_m)^{0}_!"]  \\
 K_0(C^*(\G(\Sigma)^{\mathrm{ad}})) \arrow{r}{\mathrm{ev}_{t=0}}&  K_0(C^*(\A_\Sigma))  
\end{tikzcd}
\end{equation}  
Here $\G(Z)= \presuper{\varphi_m}(\G)$ and $\G(\Sigma)=\presuper{\varphi}\G$. The horizontal maps are isomorphisms. Proceeding as in \cite{DLN}, we prove that   $(\pi_m)^{0}_!$ is the Thom isomorphism of the $\Spinc$-bundle $\A_Z  \longrightarrow \A_\Sigma $.  Then  $(\pi_m)_!^{\mathrm{ad}}$ is an isomorphism too.


Let $D_\Sigma$ be the Dirac operator of the cycle $x$, let $S_{Z}$ be the complex spinor bundle  of $VZ$ and set $E_{Z}=E\otimes S_{Z}$. By \cite[Propositions 3.6 and 3.11]{BHS} we can choose a  Dirac bundle  $(E_{Z},D_{Z})$ such that 
\[
(\pi_m)^{0}_![D_Z)]_{\pr,\ev}=[D_\Sigma]_{\pr,\ev}. 
\]
We then deduce from the previous diagram that
\[
 (\pi_m)^{\mathrm{ad}}_!([D_Z]_{\ad,\ev})=[D_\Sigma]_{\ad,\ev}\in K_0(C^*(\G(\Sigma)^{\mathrm{ad}})).
\]
Now, since $D_\Sigma$ is tameable, we obtain by using the commutative diagram: 
{\small \begin{equation}\label{cd:bound}
\begin{tikzcd}
 K_1(C^*(\G(Z)|_{\partial Z})) \arrow{r}{\partial}\arrow[d, "\M", "\simeq" ']  &  K_0(C^*(\G(Z)^\F)) \arrow{r}\arrow[d, "(\pi_m)^\F_!", "\simeq" ']  & K_0(C^*(\G(Z)^{\mathrm{ad}}))   \arrow{r}{\mathrm{ev}_{\partial Z\times 1}} \arrow[d, "(\pi_m)^{\mathrm{ad}}_!", "\simeq" ']  & K_0(C^*(\G(Z)|_{\partial Z})) \arrow[d, "\M", "\simeq" ']  \\
 K_1(C^*(\G(\Sigma)|_{\partial \Sigma})) \arrow{r}{\partial}&  K_0(C^*(\G(\Sigma)^\F)) \arrow{r}&  K_0(C^*(\G(\Sigma)^{\mathrm{ad}})) \arrow{r}{\mathrm{ev}_{\partial \Sigma\times 1}} & K_0(C^*(\G(\Sigma)|_{\partial \Sigma}))  .
\end{tikzcd}
\end{equation}} 
that $D_Z$ is $\partial Z$-tameable too and we pick up a $\partial Z$-taming $B_{Z}$  such that
\[
(\pi_m)^{\F}_![B_Z]=[B]
\] 
The constructions are similar for odd geometric cycles.

\begin{Def}\label{Def:vb-modif-geocycle}\ 
 A vector bundle modification of a geometric cycle $x=(\Sigma, \varphi, E, B)$ over $(M,\G)$ by an even rank $\Spinc$-vector bundle $\pi: V\to \Sigma$  is a  geometric cycle $m(x,V)=(Z,\varphi_m, E_m, B_Z)$ over $\G(M)$ such that
\begin{enumerate} 
 \item  $Z= S(V \oplus 1_{\Rr})$ is the sphere bundle;
 \item $\varphi_m= \varphi\circ \pi_m : Z \longrightarrow M$;
 \item $\pi_!^{\nc}[B_Z]_{\partial Z,\ast}=[B]_{\partial \Sigma,\ast}\in K_{\ast}(C^{*}\Tau\Sigma))$. 
\end{enumerate}
\end{Def}

\subsection*{Geometric $K$-homology}

\begin{Def}
The equivalence relation $\sim$ on $ \EVgeo_*(M)$ is the one generated by the following operations:
\begin{enumerate}[(i)]
 \item Isomorphisms of geometric cycles;
 \item Direct sums: if  $x_i=(\Sigma, \varphi, E_i, B_i)$, $i=1,2$ are geometric cycles  then 
  \begin{equation}
   x_1\cup x_2 \sim (\Sigma, \varphi, E_1\oplus E_2, B_1\oplus B_2);
  \end{equation}
 \item Cobordisms; 
 \item Vector bundle modifications.
 \end{enumerate}
The quotient set 
\[
\KVgeo_*(M) :=  \EVgeo_*(M)  / \sim 
\]
is called geometric $K$-homology of $(M,\G)$. 
 \label{Def:Kgeo}
\end{Def}
 
\begin{Thm}
The following formula: 
\begin{equation}\label{eq:add-Kgeo}
  \forall x_0,x_1\in \EVgeo(M),\qquad [x_0]+[x_1] = [x_0\cup x_1].
\end{equation}
turns  $(\KVgeo(M),+)$ into an abelian group.  
\label{Thm:Kgeo}
\end{Thm}
\begin{proof}
If the equivalence $x_{0}\sim x_{0}'$ is given by one of the four elementary operations of Definition \ref{Def:Kgeo} then $x_{0}\cup x_{1}\sim x_{0}'\cup x_{1}$. Therefore \eqref{eq:add-Kgeo} is well defined. Commutativity and associativity of $+$ is obvious. The neutral element is represented by $0_x$ and $-[x] = [-x]$ for any $x\in \EVgeo(M)$: details are provided in Example \ref{Ex:easy-cobordisms}. 
\end{proof}

\subsection*{Comparison map}

\begin{Thm}\label{Thm:lambda}
 The map $(\Sigma,\varphi,E,B)\in \EVgeo_{\ast}(M)\longmapsto \varphi^{\mathrm{nc}}_![B]_{\partial \Sigma, \ast}$ gives rise to a well defined homomorphism:
 \begin{equation}
  \lambda \colon \KVgeo_{\ast}(M) \to K_{\ast}(C^*(\Tau M)).
 \end{equation}
\end{Thm}
We emphasize that this result, together the conclusion of Theorem \eqref{Thm:pushforward-respect-indices}, gives information about the invariance of the Fredholm index of tame Dirac operators. For instance,  let us consider even tame Dirac operators $B_j$ associated with integrated Lie manifolds $(\Sigma_j,\G_j)$, $j=1,2$. If $(\Sigma_1,B_1)$ and $(\Sigma_2,B_2)$ can be enhanced into cobordant geometric cycles over some integrated Lie manifold $(M,\G)$, then their Fredholm index coincide: 
$$\ind^\F([B_1]_{\partial \Sigma_1,\ev})= \ind^\F([B_2]_{\partial \Sigma_2,\ev}).$$
This can be viewed as a generalization of the cobordism invariance of Dirac operators on closed manifolds. The comparison homomorphism $\lambda$ provides the geometric $K$-homology with a split surjection to the $K$-group of symbols. But it is not shown here that this map is injective, which it is in the classical Baum-Douglas case. We do not need the injectivity in this article. However, the question under which conditions $\lambda$ is an injection, is to be investigated in future research.

\begin{proof}
We prove the even case, the odd one is similar.

\emph{Invariance under cobordism:} 
 Let $w=(W, Y, \Phi, E, B)$ be a cobordism between two even geometric cycles $x_i=(\Sigma_i, \varphi_i, E_i, B_i)$, $i=0,1$ over $(M,\G)$. We denote by:
 \begin{enumerate}
 \item $\G(W)=\presuper{\Phi_b}\G$ the groupoid associated with the cobordism $w$;
  \item $\G_i=\presuper{\varphi_i}\G$ the groupoid associated with $x_i$, $i=0,1$;
 \item $\rho_{i}:=\rho_{\Sigma_{i}}: C^*(\Tau_YW)\to  C^*(\Tau\Sigma_i\times\Rr)$ the homomorphism defined just after Diagram \eqref{cd:cobord-2-boundary}.
   \end{enumerate}


$\bullet$ Recall that we  have by definition of a cobordism of geometric cycles: 
\[
\rho_i[B]_{Y,\odd}=(-1)^i[B_i]_{\partial\Sigma_{i},\ev}\otimes\beta
\]
where $\beta$ is the chosen Bott generator of $K_1(C^*(\Rr))$.

\medskip $\bullet$ We are going to prove that 
 \begin{equation}\label{eq:push-equal} 
(\varphi_{0})^{\nc}_!([B_0]_{\partial\Sigma_{0},\ev})=(\varphi_{1})^{\nc}_!([B_1]_{\partial\Sigma_{1},\ev})
\end{equation}
 where  the homomorphisms are defined in \eqref{eq:push-forward-nct}.

Denote by $J_i$ the kernel of $\rho_{i} $ and observe that $J_{0}+J_{1}=C^*(\Tau_YW)$. Set $J=J_{0}\cap J_{1}$ and consider the exact sequences:
\begin{equation}\label{eq:1-boundary}
 0\longrightarrow J \longrightarrow J_{i} \overset{\rho_i}{\longrightarrow} C^*(\Tau\Sigma_i\times\Rr) \longrightarrow 0. 
\end{equation}
for $i=0,1$. We denote by $\partial_{i}\in KK_{1}(C^*(\Tau\Sigma_i\times\Rr), J_{i})$ the associated boundary elements.  Consider also:
\begin{equation}\label{eq:2-boundary}
 0\longrightarrow J \longrightarrow C^*(\Tau_YW) \overset{\rho_{0}\oplus \rho_{1}}{\longrightarrow} C^*(\Tau\Sigma_0\times\Rr)\times C^*(\Tau\Sigma_1\times\Rr)\longrightarrow 0
\end{equation}
whose ideal is given by $J=J_{0}\cap J_{1}$ and boundary element denoted by $\partial$.
By  \cite[Lemma 3.5]{VZ2016}) we have:
\begin{equation}\label{eq:splitting-boundary}
 \partial(x_{0},x_{1}) = \partial_0(x_{0})+ \partial_1(x_{1}) \in KK_1(\Cc,J),\quad x_{i}\in KK_{1}(\Cc,C^*(\Tau\Sigma_i\times\Rr)).
\end{equation}

We know that 
\begin{equation}
 \rho_{0}\oplus \rho_{1}([B]_{Y,\odd}) = \left([B_0]_{\partial\Sigma_{0},\ev}\otimes\beta,-[B_1]_{\partial\Sigma_{1},\ev}\otimes\beta \right),
\end{equation}
therefore formula \eqref{eq:splitting-boundary} and exactness imply:
\begin{equation}\label{eq:pre-inv-cobordism}
 \partial_0([B_0]_{\partial\Sigma_{0},\ev}\otimes\beta)= \partial_1([B_1]_{\partial\Sigma_{1},\ev}\otimes\beta) \in K_0(J).
\end{equation}

It is now time to extend the deformation \eqref{eq:thomlike}: 
\begin{equation}\label{eq:thomlike-the-big-one}
 \L_w= (\G(W))^{\mathrm{ad}}\times\{u=0\} \cup \presuper{(\Phi_1)_b}(\G^{\mathrm{ad}}) \times(0,1]_u  \rightrightarrows W\times[0,1]_t \times[0,1]_u.
\end{equation}
Here $\Phi_1 =\Phi\circ \pr_{1}: W\times [0,1]_t\to M$. We can consider various saturated sub-groupoids of $\L_w$. For instance: 
\begin{equation}
 \L_Y^{\nc}(w) = \L_w|_{(W\times\{t=0\}\cup Y\times(0,1)_{t})\times [0,1]_{u}}.
\end{equation}
If we consider in $\L_Y^{\nc}(w)$ the faces corresponding to $\Sigma_i$, we recover the groupoids $\L^\nc_{\varphi_i}\times\Rr$. Continuing in this way and 
denoting by $\Phi' : W\setminus(\Sigma_0\cup\Sigma_1)\to M$ the restriction of $\Phi$ (which is again a surjective submersion), we get the following commutative diagram:
\begin{center}
\begin{equation}\label{cd:bound2}
\begin{tikzcd}[row sep=huge, column sep=huge, text height=1.5ex, text depth=0.25ex]
 K_1(C^*(\Tau \Sigma_0\times\Rr))\oplus K_1(C^*(\Tau \Sigma_1\times\Rr)) \arrow{r}{\partial}&  K_0(J)  \\
  K_1(C^*(\L_{\varphi_0}^\nc\times\Rr))\oplus K_1(C^*(\L_{\varphi_1}^\nc\times\Rr))\arrow[u, "\simeq", "u=0"']  \arrow{d}{u=1}\arrow{r}{\partial'}&  K_0(...)  \arrow[u, "\simeq", "u=0"']\arrow{d}{u=1} \\
 K_1(C^*(\presuper{\varphi_1}(\Tau M\times\Rr)))\oplus K_1(C^*(\presuper{\varphi_1}(\Tau M\times\Rr)))\arrow[d, "\M", "\simeq" '] \arrow{r}{\partial''}&  K_0(C^*(\presuper{\Phi'_1}(\Tau M)))\arrow[d, "\M", "\simeq" '] \\
 K_1(C^*(\Tau M\times\Rr))\oplus K_1(C^*(\Tau M\times\Rr))  \arrow{r}{\partial'''}&  K_0(C^*\Tau M))  
\end{tikzcd}
\end{equation}
\end{center}
The map in the bottom line is given by addition and Bott periodicity, in particular, if $\partial'''(u\oplus v)=0$ then $u=-v$. 

The map obtained from top to bottom in the left column is equal to $(\varphi_0)_!^{\nc}\otimes \Id \oplus (\varphi_1)_!^{\nc}\otimes \Id$. 

Therefore, using the equality \eqref{eq:pre-inv-cobordism} together with the commutativity of the previous diagram and the remarks just above, we conclude that 
the equality \eqref{eq:push-equal} holds true.

\bigskip \emph{Invariance under vector bundle modification:} 

Consider $x = (\Sigma, \varphi, E, \Psi) \in \EVgeo(M)$ and a vector bundle modification $m(x, V) = (Z, \varphi_m, E_m, B_Z)$. 
One one hand, we know that 
\[
  (\pi_m)^{\F}_!([B_Z])= [B].
\]
On the other hand, we have $\varphi_m=\varphi\circ \pi_m$, therefore:
\[
   (\varphi_m)^{\F}_!([B_Z]) = (\varphi\circ \pi_m)^{\F}_! =(\varphi^{\F}_! \circ (\pi_m)^{\F}_!)([B_Z])= \varphi^{\F}_!([B])
\]
and the equality $(\varphi_m)^{\F}_!([B_Z])=\varphi^{\F}_!([B])$ proves the invariance under vector bundle modification. 
\end{proof}

\section{Reduction to Callias-type operators}
\label{RedDirac}

Let $(M,\G)$ be an integrated Lie manifold. We are going to define a  \emph{clutching map}\footnote{We only treat here the clutching map in the even case, the odd case is similar.}:
\begin{equation}
\widetilde{c} \colon \Gamma(\mu_{\partial M}) \to \KVgeo_0(M)
\end{equation}
and show that it descends to  a  \emph{clutching homomorphism} $c\colon  \FEllV(M)\to \KVgeo_{0}(M) $. 

We set $\Sigma_{\A} = S(\A \oplus \Rr)$ for the clutching space of $\A$. We denote by $\varphi$ the tame submersion $\SigmaA\to M$. 
We can write, $\Sigma_{\A} = \B(\A)_{+} \cup \B(\A)_{-}$, where $\B(\A)$ denotes the ball bundle in $\A$ and $\B(\A)_{\pm}$ denote the upper and lower hemispheres respectively. Let $\widehat{\A} = \A \cup S(\A)$ be the radial compactification of $\A$ and $\widehat{\pi} \colon \widehat{\A} \to M$ the corresponding projection map.

Let  $P \in \Psi_{\V}^m(M; E_0, E_1)$ be a  fully
elliptic operator. By ellipticity of $P$ we have an isomorphism
\[
\sigma_{\mathrm{pr}}(P) \colon \pi^{\ast} E_0 \iso \pi^{\ast} E_1. 
\]
The clutched bundle $E_{\sigma} \to \SigmaA$ is defined by the glueing of pullbacks of $E_0$ and $E_1$, along the boundary stratum $S(\A)$, using $\sigma_{\mathrm{pr}}(P)$:
\[
E_{\sigma} = \widehat{\pi}^{\ast} E_0 \cup_{S(\A)} \widehat{\pi}^{\ast} E_1.
\]
To define a geometric cycle $\widetilde{c}(\E_0, \E_1, \sigma_f(\Lambda^{-1} P))$ associated with $(\E_0, \E_1, \sigma_f(\Lambda^{-1} P))\in \Gamma(\mu_{\partial M})$, we consider the Dirac operator $D_{\SigmaA}$ on the integrated Lie manifold $(\SigmaA,\Gphi)$ associated with the  $\Spinc$-structure of  $\Aphi\to \Sigma_\A$, and we first observe:
\begin{Lem}(\cite{BVE})
Let $(M,\G)$ be an integrated Lie manifold and $P\in \Psi_{\G}(E_0,E_1)$ be an elliptic operator. Then: 
\begin{align}
& (i\circ\mathrm{Thom})([P]_{\pr,\ev})=  [D_{\SigmaA}]_{\pr,\ev}\otimes ([E_\sigma]-[\varphi^*E_1])\in K^0(\Aphi). \label{L}
\end{align}
\end{Lem}
This lemma is proved in \cite{BVE} in the case $\A=TM$. The proof is the same here. Secondly:
 \begin{Thm}\label{Prop:taming-clutching}
 Let $(M,\G)$ be an integrated Lie manifold and $P\in \Psi_{\G}^{m}(E_0,E_1)$ be a fully elliptic operator. Then there exists an even Dirac bundle $(E,D)$ on $(\Sigma_\A,\Gphi)$ with boundary taming $B$ such that 
 \begin{equation}\label{eq:taming-clutching}
   \varphi_!^{nc}([B]_{\partial \SigmaA,\ev}) = [P]_{\partial M,\ev} \in K_0(C^*(\Tau M)).
 \end{equation}
 In particular,  $B_{+}$ and $P$ have the same Fredholm index. 
\end{Thm}

Let us remark that when $\A$ is Spin${}^c$, the tame Dirac bundle can be obtained on the integrated Lie manifold $(M,\G)$ itself.  Indeed, using the Thom isomorphism: $K^0(M) \simeq K^0_c(\A^{*})$, we see that the principal symbol  of any fully elliptic operator is in the same class as the one of a Dirac bundle $(M,\G)$, and we then can apply Theorem \ref{thm:diracification}. 

\begin{proof}
 Rewriting the proof of Theorem \ref{Thm:Thom-Morita-Kfunctor}, we get a commutative diagram:
\begin{center}
\begin{equation}\label{D0}
\begin{tikzcd}[row sep=huge, column sep=huge, text height=1.5ex, text depth=0.25ex]
K^0_c(\A^{\ast})\arrow{r}{\mathrm{Thom}} & K^0_c({}^\pi\!\A^{\ast}) \arrow{r}{- \otimes [i]} & K^0_c({}^\varphi\!\A^{\ast})\\
K_0(C^{\ast}(\Tau M)) \arrow{u}{\mathrm{ev}_{t=0}} \arrow{r}{(\pi_!^{nc})^{-1}} & K_0(C^{\ast}(\Tau \A)) \arrow{u}{\mathrm{ev}_{t=0}}  \arrow{r}{- \otimes [i]} & K_0(C^*(\Tau\SigmaA))\arrow{u}{\mathrm{ev}_{t=0}}   
\end{tikzcd}
\end{equation}
\end{center}
By the previous lemma and this diagram, we obtain the existence  of a fully elliptic $Q$ on $\Gphi$ such that 
\begin{equation}
 [Q]_{\pr,\ev} = [D_{\SigmaA}]_{\pr,\ev}\otimes ([E_\sigma]-[\varphi^*E_1]) = [D_1]_{\pr,\ev}
\end{equation}
where $(D_1,E_\sigma\oplus \varphi^*E_1)$ is the even Dirac bundle with:
\[
 D_1=  D_{\SigmaA} \otimes E_\sigma \oplus D_{\SigmaA}^{op} \otimes \varphi^*E_1  
\]
and $D_{\SigmaA} \otimes E_\sigma$ is the Dirac operator on $\Sigma$ twisted by $E_\sigma$ and $D_{\SigmaA}^{op}=-D_{\SigmaA} $. 
Now, using Theorem  \ref{Thm:26}, we get an even Dirac bundle $(\widetilde{E},\widetilde{D})$ with boundary taming $B=\widetilde{D}+R$ such that we still have 
\[
 [B]_{\partial \SigmaA,\ev} = [Q]_{\partial \SigmaA,\ev}\in K_{0}(C^*(\Tau \SigmaA)),
\]
from which we conclude that \eqref{eq:taming-clutching} holds true using Theorem \ref{Thm:Thom-Morita-Kfunctor}, (2).
The last assertion comes from Theorem \ref{Thm:pushforward-respect-indices}.
\end{proof}

Now, using Theorem \ref{Prop:taming-clutching} to pick up suitable $(\widetilde{E},\widetilde{D})$ and $B$, we define:
\[
\widetilde{c}(\E_0, \E_1, \sigma_{\partial M}(\Lambda^{-m} P)) = [(\SigmaA, \widetilde{E}, \varphi, B)]_{iso} \in  \KVgeo_0(M). 
\]

The clutching map $\widetilde{c}$ sends a relative cycle to an isomorphism class of geometric cycles (note that two choices of $B$ yield isomorphic cycles). 

\begin{Thm}
The map $\widetilde{c}$ induces a well defined homomorphism
\[
c \colon  \FEllV(M) \to \KVgeo_0(M). 
\]
\label{Prop:clutching}
\end{Thm}
\begin{proof}

Let $(P_0)_{rel} := (\E_0, \F_0,  p_{0})$ and $(P_1)_{rel} := (\E_1, \F_1, p_{1})$
be elements of $\Gamma_0(\mu_M)$ which are equivalent. Let  $\alpha,\beta\in \Gamma_0(\mu_M)$ be elementary elements such that 
\[
(\E_0, \F_0,  p_{0})\oplus \alpha \simeq  (\E_1, \F_1, p_{1})\oplus \beta. 
\] 
Without loss of generality, we can consider that $E=\E_{0}=\E_1$, $F=\F_{0}=\F_1$ and $p_{0}$, $p_{1}$ connected by a smooth homotopy $p_{t}$ of fully elliptic symbols. Denote by $\sigma_{t}: \pi^{*}E \longrightarrow \pi^{*} F$ the map obtained by extending by homogeneity $\sigma_{\pr}(P_{t})$ over $\A\setminus \{0\}$ and multiplying it with a function $\chi \in C^{\infty}(\A)$ such that $\chi(\xi)= 0$ in a neighborhood of the $0$-section and $\chi(\xi)=1$ near infinity.  

Define $\widehat{E}$ via the glueing diffeomorphism
$[0,1] \times S(\A) \times E \to [0,1] \times S(\A) \times F, \ (t, x, v) \mapsto (t, x, \sigma_t(v))$ along the cylinder
$[0,1] \times S(\A)$. This furnishes the vector bundle $\widehat{E}$ over the cylinder clutching space $\widehat{\Sigma} := [0,1]\times \Sigma$. 
Denote by $\widehat{\varphi} : \widehat{\Sigma} \longrightarrow M$ the natural projection. Setting $Y= [0,1]\times \partial \Sigma$, we already get a cobordism 
over $M$, namely $(\widehat{\Sigma},\widehat{\varphi},Y)$. 

Let $D_{j}$ be the Dirac operator underlying the geometric cycle $\widetilde{c}((P_j)_{rel})$. We consider on $\widehat{\Sigma}$ the $\Spinc$-structure that coincides with the one of $\Sigma$ at $\Sigma\times \{0\}$ and to the opposite one at $\Sigma\times \{1\}$. 

Let us consider a Dirac operator $\widehat{D}$ on the integrated Lie manifold $(\widehat{\Sigma}, \presuper{\widehat{\varphi}_{b}}\G)$ whose boundary reductions at $\Sigma\times \{j\}$ coincide with $D_{j}$.  The corresponding Clifford vector bundle over $\widehat{\Sigma}$ is: 
\[
 \widetilde{E} = S_{\widehat{\Sigma}}\otimes \widehat{E} \oplus  (-S_{\widehat{\Sigma}})\otimes \widehat{\varphi}^{*}F. 
\]

Arguing as in Example \ref{Ex:easy-cobordisms}, we obtain a $Y$-taming $\widehat{B}$ such that $(\widehat{\Sigma},\widehat{\varphi},Y, \widetilde{E},\widehat{B})$ is a cobordism between the geometric cycles $\widetilde{c}((P_j)_{rel})$.

\end{proof}

\begin{Thm}
The diagram
\begin{figure}[H]
\begin{tikzcd}[row sep=huge, column sep=huge, text height=1.5ex, text depth=0.25ex]
K_{\ast}(\mu) \arrow{d}{c} \arrow{dr}{\pd} & \\
\KVgeo(M) \arrow{r}{\lambda} & K(C^{\ast}(\Tau M)) 
\end{tikzcd}
\end{figure}
commutes. In particular, if $P \colon C^{\infty}(M, E_0) \to C^{\infty}(M, E_1)$ denotes a fully elliptic pseudodifferential operator on the integrated Lie manifold $(M, \G)$, then there is a taming $\C$ on the integrated Lie manifold $(\SigmaA, \Gphi)$ such that $\ind(\C) = \ind(P)$. 
\label{Thm:redDirac}
\end{Thm}
The proof follows from Theorem \ref{Prop:taming-clutching}. The injectivity of the comparison map $\lambda$ will be studied elsewhere.

\appendix

\section{Transitivity of cobordisms}

\begin{Thm}\label{Lem:transitivity-cobordisms}
The cobordism relation is transitive up to isomorphisms of geometric cycles. 
\end{Thm}
\begin{proof}
Let $x_i=(\Sigma_i, \varphi_i, E_i, B_i)$, $i=1,2,3$, be  geometric cycles over $(M,\G)$ such that $x_1$ is cobordant to $x_2$  and $x_2$ cobordant to $x_3$.  
 Let $w_i=(W_i, Y_i, \Phi_i, F_i, C_i)$, $i=1,2$ be respective cobordisms. 
We set $\Phi=\Phi_1\cup\Phi_2 : W=W_1\underset{\Sigma_2}{\cup}W_2 \longrightarrow M$, $Y=Y_1\cup Y_2$. The triple $(W,\Phi,Y)$ is a manifold cobordism over $M$ between $\Sigma_1$ and $\Sigma_3$. We set : 
\begin{equation}
    \widetilde{\G}(W) = \G(W_1) \bigcup_{\G(\Sigma_2)\times \Rr} \G(W_2).
\end{equation}
Also, setting $F=F_1\underset{\Sigma_2}{\cup}F_2$ and using the point (3) in Definition \ref{Def:bordism-geocycle}, we get a Dirac bundle $(F,\widetilde{D})$ over $(W,  \widetilde{\G}(W))$ that restricts to the one of $w_{j}$ on $W_{j}$, $j=1,2$.  In particular, we have by assumption that $(F,\widetilde{D})$  is $Y$-tameable. 

Note that $\widetilde{\G}(W)= \SBlup_{r,s}(\G(W),\G(W)_{\Sigma_{2}}^{\Sigma_{2}})$ where 
\begin{equation}
  \quad\G(W)=\presuper{\presuper{b}\Phi}\G(M).
\end{equation}
Choosing a collar decomposition around $\Sigma_2$ into $W$, we get a deformation Lie groupoid:
\begin{equation*}
 \H(W) = \{t=0\}\times  \widetilde{\G}(W)\cup (0,1]_t\times \G(W)\rightrightarrows [0,1]\times  W
\end{equation*}
and it is clear that the Dirac bundle  $(F,\widetilde{D})$  can be lifted to a Dirac bundle $(F,C)$ over $(W\times [0,1],\H(W))$ and we thus end with a Dirac bundle $(F,D)$ over $(W,\G(W))$ by restriction at $t=1$. 

Since $\widetilde{D}$ is $Y$-tameable, we get using the following diagram (where we use this time Fredholm groupoids rather than their $K$-equivalent non commutative tangent space counterparts) that $C$ is $Y\times [0,1]$-tameable too and finally $D$ is $Y$-tameable. 

{\small \begin{equation}\label{cd:Y-taming}
\begin{tikzcd}
   K_0(C^*(\widetilde{\G}(W)^\F_Y)) \arrow{r}   & K_0(C^*(\widetilde{\G}(W)^{\mathrm{ad}}))     \\
 K_0(C^*(\H(W)^\F_{[0,1]\times Y})) \arrow{r}\arrow[d, "|_{t=1}"] \arrow[u, "|_{t=0}"']  & K_0(C^*(\H(W)^{\mathrm{ad}}))  \arrow[d, "|_{t=1}"]  \arrow[u, "|_{t=0}"', "\simeq"]   \\
   K_0(C^*(\G(W)^\F_Y)) \arrow{r}&  K_0(C^*(\G(W)^{\mathrm{ad}})) .
\end{tikzcd}
\end{equation}} 
Fixing a $Y$-taming $B$ for $D$ then provides the required cobordism $(W,Y,\Phi,F,B)$ between $x_{1}$ and $x_{3}$. 
\end{proof}

\section*{Acknowledgements}
For useful discussions we thank Bernd Ammann, Ulrich Bunke, Paulo Carrillo Rouse, Claire Debord, Magnus Goffeng, Victor Nistor, Elmar Schrohe, Georges Skandalis and Georg Tamme.
The first author was supported by the DFG-SPP 2026 `Geometry at Infinity'. The second author was supported by the Grant ANR-14-CE25-0012-01 SINGSTAR.
We are also happy to thank the referees for useful questions, comments and corrections.

\bigskip 
Karsten Bohlen, {\sc Universit\"at Regensburg, 93040 Regensburg, Germany} 

\emph{Email Address : } \textbf{karsten.bohlen@mathematik.uni-regensburg.de}

\bigskip Jean-Marie Lescure, {\sc Universit\'e Paris Est Creteil, Univ Gustave Eiffel, CNRS, LAMA UMR8050, F-94010 Creteil, France} 

\emph{Email Address : } \textbf{jean-marie.lescure@u-pec.fr}  

\end{document}